\documentclass[final]{siamltex}

\usepackage{amsmath}
\usepackage{amsbsy}
\usepackage{amssymb}
\usepackage{textcomp}
\usepackage{graphicx}
\usepackage{amsfonts}
\usepackage{longtable}
\usepackage{multirow}
\usepackage{lscape}
\usepackage{url}
\usepackage{subfig}
\usepackage{longtable}
\usepackage{lscape}
\usepackage{hhline}
\usepackage{colortbl,array} 
\usepackage{multirow,bigdelim}
\usepackage{booktabs}
\usepackage[ruled]{algorithm2e}
\usepackage{epstopdf}

\newtheorem{thm}{Theorem}[section]

\newtheorem{prop}[thm]{Proposition.\nopagebreak}

\newtheorem{exam}[thm]{Example.\nopagebreak}

\def\eeq{\end{equation}} 
\def\lbeq#1{\begin{equation} \label{#1}} 

\def\Forall{~~~\mbox{for all }}  
\def\eps{\varepsilon} 
\def\D{\displaystyle}				
\def\ol{\overline}

\def\fct#1{\mathop{\rm #1}}	

\def\argmin{\fct{argmin}}

\def\half{\frac{1}{2}}
\def\wh{\widehat}

\def\bary{\begin{array}}
\def\eary{\end{array}}

\def\st{\fct{s.t.~}}

\def\ol{\overline}
\def\<{\langle}
\def\>{\rangle}
\def\Rz{\mathbb{R}}
\def\Forall{{~~~\mbox{for all}~}}

\def\half{\frac{1}{2}} 

\title{An optimal subgradient algorithm for large-scale convex optimization in simple domains
       }

\author{
Masoud Ahookhosh\thanks{Faculty of Mathematics, University of Vienna,
Oskar-Morgenstern-Platz 1, 1090 Vienna, Austria. ({\tt masoud.ahookhosh@univie.ac.at})}
\and Arnold Neumaier\thanks{Faculty of Mathematics, University of Vienna,
Oskar-Morgenstern-Platz 1, 1090 Vienna, Austria. ({\tt Arnold.Neumaier@univie.ac.at})}
        }
\begin{document}
\maketitle
\slugger{mms}{xxxx}{xx}{x}{x--x}

\begin{abstract}
This paper shows that the optimal subgradient algorithm, OSGA, proposed in \cite{NeuO} can be used for solving structured large-scale convex constrained optimization problems. Only first-order information is required, and the optimal complexity bounds for both smooth and nonsmooth problems are attained. More specifically, we consider two classes of problems: (i) a convex objective with a simple closed convex domain, where the orthogonal projection on this feasible domain is efficiently available; (ii) a convex objective with a simple convex functional constraint. If we equip OSGA with an appropriate prox-function, the OSGA subproblem can be solved either in a closed form or by a simple iterative scheme, which is especially important for large-scale problems. We report numerical results for some applications to show the efficiency of the proposed scheme. A software package implementing OSGA for above domains is available. 
\end{abstract}



\pagestyle{myheadings}
\thispagestyle{plain}
\markboth{\sc M. Ahookhosh and A. Neumaier}{\sc OSGA for  convex optimization in simple domains}

\section{Introduction}
Convex optimization has been shown to provide efficient algorithms for computing reliable solutions in a broad range of applications. Many applications arising in applied sciences and 
engineering such as signal and image processing, machine learning, statistics, and general inverse problems can be addressed by a convex optimization problem involving high-dimensional data. In practice, solving a nonsmooth convex problem is usually more difficult and costly than a smooth one. More precisely, for a prescribed accuracy parameter $\varepsilon$, the optimal complexity to achieve an $\varepsilon$-solution of nonsmooth Lipschitz continuous problems is $O(\varepsilon^{-2})$, the superior complexity $O(\varepsilon^{-1/2})$ for smooth problems with Lipschitz continuous gradient, see \cite{NemY,NesB}.

Thanks to the low memory requirement and simple structure, first-order methods have received much attention during the past few decades. Indeed, they deal successfully with large-scale problems. In general,
convex optimization problems can be solved by gradient-type algorithms \cite{AhoG,BerB,BerT,Gol}, conjugate gradient methods \cite{GilN,HagZ1,HagZ2} and spectral gradient methods \cite{BarB,BirMR,Ray} for smooth objectives and by subgradient-type methods \cite{BoyXM,NedB,NesP}, proximal gradient methods \cite{ParB,ComP}, smoothing techniques \cite{BecT1,BotH1,DevGN2,NesS}, bundle-type algorithms \cite{LamNN,Lan}, and primal-dual first-order methods \cite{BotH2,BotCH,ChaP} for nonsmooth objectives.
Moreover, both classes can be addressed by (zero-order) coordinate descent methods and derivative-free methods. The current paper only addresses first-order methods and assumes that first-order black-box information -- function values and subgradients -- of the objective function are available.

Historically, gradient descent and subgradient methods were the first numerical schemes proposed to solve optimization problems 
with smooth and nonsmooth convex objective functions, respectively. In practice, they are too slow, especially for badly scaled problems. This can be addressed by their worst-case complexity bounds to reach an $\varepsilon$-solution, while the gradient descent method achieve the 
complexity of the order $O(\varepsilon^{-1})$ which is not optimal for smooth problems, the subgradient methods attain the worst-case complexity of the order $O(\varepsilon^{-2})$.
In 1983, {\sc Nemirovski \& Yudin} in \cite{NemY} derived optimal worst-case complexity bounds of first-order methods to achieve an $\varepsilon$-solution for several class of problems such as Lipschitz continuous nonsmooth problems and smooth problems with Lipschitz continuous gradient. If an algorithm attains the optimal worst-case complexity bound for a class of problems, it is called optimal. Optimal first-order methods dating back to {\sc Nesterov} \cite{Nes83} in 1983. This optimal first-order method is interesting both theoretically and computationally, attracting many researchers to work in the development of such schemes, for example {\sc Auslander \& Teboulle} \cite{AusT}, {\sc Beck \& Teboulle} 
\cite{BecT2}, {\sc Devolder} et al. \cite{DevGN1}, {\sc Gonzaga} et al. \cite{GonK,GonKR}, {\sc Lan} \cite{Lan}, {\sc Lan} et al. \cite{LanLM}, {\sc Nesterov} 
\cite{NesS,NesE,NesC}, {\sc Neumaier} \cite{NeuO} and 
{\sc Tseng} \cite{Tse}. Computational 
comparisons for composite functions show that optimal Nesterov-type first-order methods are substantially superior to the gradient descent and subgradient methods, see, for example, {\sc Ahookhosh} \cite{Aho} and {\sc Becker} et al. \cite{BecCG}.\\

{\bf Content.} In this paper we consider structured convex constrained optimization problems frequently observed in applications and develop OSGA to efficiently solve such problems. Two clasess of convex domains are considered, namely, simple convex domains such that the orthogonal projection is cheaply feasible, and sublevel set of a convex function referred as functional domain. For problems with a simple domain, we first introduce an appropriate prox-function and then show that the solution of OSGA's subproblem is obtained by a projection on the domain followed by solving a one-dimensional nonlinear equation. It is shown that if explicit formula for projection is available, the nonlinear equation can be solved in a closed form in many interesting cases. We also establish the optimality condition for functional domain and show for some simple functions that results to in a closed form solution. Finally, we report some numerical results for applications to show the efficiency OSGA in comparison with some state-of-the-art algorithms. 

The remainder of this paper is organized as follows. In the next section, we review the basic idea of OSGA. Section 3 considers the structured convex constrained
minimization and how to solve the associated OSGA subproblem. We report numerical results in Section 4 and  our conclusions are derived in Section 5.\\

{\bf Notation and preliminaries.}
Let $\mathcal{V}$ be a real finite-demensional vector space endowed with the norm $\|\cdot\|$, and $\mathcal{V}^*$ denotes the dual space of all linear functional on $\mathcal{V}$ where the bilinear pairing  $\langle g,x\rangle$ denotes the value of the functional $g \in \mathcal{V}^*$ at $x \in \mathcal{V}$. If $\mathcal{V} = \mathbb{R}^n$, then 
\[
\|x\|_2 := \left(\sum_{i=1}^n |x_i|^2\right)^{1/2}.
\]
If $x \in \mathbb{R}^{m \times n}$, then the Schatten $\infty$-norm is $\|\sigma(x)\|_\infty$ where $\sigma: \mathbb{R}^{m \times n} \rightarrow \mathbb{R}^{\min \{m,n\}}$ is the function that takes a matrix $x \in \mathbb{R}^{m \times n}$ and returns a vector of singular values in nonincreasing order. If $x$ is a positive definite matrix, we denotes it by $x \succcurlyeq 0$. We also denote by $x =  \sum_{i=1}^n \lambda_i u_i u_i^T$ and $x =  \sum_{i=1}^n \sigma_i u_i v_i^T$ the eigenvalue decomposition and the singular value decomposition of $x$. For a function $f: \mathcal{V} \rightarrow \overline{\mathbb{R}} = \mathbb{R}\cup\{\pm \infty\}$, we denote by 
\[
\mathrm{dom}f := \{ x \in \mathcal{V} ~|~ f(x) < +\infty\}
\]
its {\bf \emph{effective domain}} and call $f$ {\bf \emph{proper}} if $\mathrm{dom}f \neq \emptyset$ and $f(x) > - \infty$ for all $x \in \mathcal{V}$. The vector $g \in \mathcal{V}^* $ is called a {\bf \emph{subgradient}} of $f$ at $x$ if $f(x) \in \mathbb{R}$ and
\[
f(y) \geq f(x) + \langle g,y-x \rangle~~~ \mathrm{for~all}~ y \in \mathcal{V}.
\] 
The set $\partial f(x)$ of all subgradients is called the 
{\bf \emph{subdifferential}} of $f$ at $x$.

We call  a nonempty, closed, and convex subset $C$ of $\mathcal V$ a 
{\bf simple convex domain} if the {\bf \emph{orthogonal projection}}
\lbeq{e.pro}
\mathrm{P}_{C}(y) := \argmin_{x \in C} \frac{1}{2} \|x-y\|^2
\eeq
of $y$ to $C$ can be found efficiently for every $y\in \cal V$. 
Note that $P_C(y)$ is unique since $\frac{1}{2} \|x-y\|^2$ is strongly 
convex. 
Computing the orthogonal projection is a well-studied topic on convex 
optimization, and the projection operator is available for many domains 
$C$ either in a closed form or by a simple iterative scheme. 
Table 1 gives some practically interesting convex domains, associated 
projection operators, and references for the formulas or iterative 
schemes.

\begin{table}[h] 
\caption{List of some available projection operators for $C=\{x \in \mathcal{V} \mid c(x)\}$}
\label{t.proj}
\begin{center}\footnotesize
\renewcommand{\arraystretch}{1.3}
\begin{tabular}{|l|l|l|}\hline
\multicolumn{1}{|l|}{{\bf defining constraint $c(x)$}} & \multicolumn{1}{l|}{{\bf Projection operator}}
&\multicolumn{1}{l|}{{\bf Ref.}} \\ 
\hline
$  Ax=b$ & $u = y - A^\dagger (Ay-b)$ & \cite{ParB}\\
\hline
$\langle a,x \rangle = b$ & $u = y - (\langle a,y \rangle - b)/(\|a\|_2^2)~ a$ & \cite{Bau}\\
\hline
$\langle a,x \rangle \leq b$ & $u = y - (\langle a,y \rangle - b)_+/(\|a\|_2^2)~ a$ & \cite{Bau}\\
\hline
$ |\langle a,x \rangle| \leq b $ & $u = \left\{
 \begin{array}{ll}
 y & \mathrm{if}~ |\langle a,y \rangle| \leq b\\
 y + (b - \langle a,y \rangle)/(\|a\|_2^2)~ a & \mathrm{if}~ \langle a,y \rangle > b\\
y + (-b - \langle a,y \rangle)/(\|a\|_2^2)~ a & \mathrm{if}~ \langle a,y \rangle < -b\\
 \end{array}
 \right.
$ & \cite{Bau,BauC}\\
\hline
$\underline{b} \leq Ax \leq \overline{b}$ & 
 $
 \begin{array}{l}
 u = x-\sum_{i=1}^N \lambda_i(x)/(\|A_{i:}\|_2^2) A_{i:}, \\
  \lambda_i(x) := \left\{
 \begin{array}{ll}
 0 &\mathrm{if}~ \underline{b}_i \leq \langle A_{i:},x \rangle \leq \overline{b}_i,\\
 \langle A_{i:},x \rangle - \overline{b}_i &\mathrm{if}~ \langle A_{i:},x \rangle > \overline{b}_i,\\
 \langle A_{i:},x \rangle - \underline{b}_i &\mathrm{if}~ \underline{x}_i > \langle A_{i:},x \rangle.
 \end{array}
 \right.
 \end{array}$ & \cite{Bau}\\
\hline
$ x \in [\underline{x}, \overline{x}]$ & $u = \sup\{\underline{x}, \inf\{y,\overline{x}\}\}
 $ & \cite{Bau}\\
\hline
$x \geq 0$ & $u = (y)_+ := \max(y,0) $ & \cite{ParB}\\
\hline
$ \|x\|_1 \leq \xi$ & iterative scheme & \cite{DucSSC,ParB}\\
\hline
$ \|x\|_2 \leq \xi$ & $u = \left\{
 \begin{array}{ll}
 \xi y/\|y\|_2 & \mathrm{if}~ \|y\|_2 > \xi\\ 
 y       & \mathrm{if}~ \|y\|_2 \leq \xi
 \end{array}
 \right.
$ & \cite{Bau}\\
\hline
$\|x\|_\infty \leq \xi$ & $u = \sup\{-\xi I, \inf\{y,\xi I \}\}
 $ & \cite{ParB}\\
\hline
$ \{(x,t) \mid \|x\|_2 \leq t \}$ & $u = \left\{
 \begin{array}{ll}
 0 & \mathrm{if}~ \|y\|_2 \leq -t\\
 (y,t) & \mathrm{if}~ \|y\|_2 \leq t\\
 1/2(1+t/\|y\|_2)(y,\|y\|_2) & \mathrm{if}~ \|y\|_2 \geq |t|\\
 \end{array}
 \right.
$ & \cite{Bau}\\
\hline
 Exponential cone & iterative scheme & \cite{ParB}\\
\hline
 Epigraphs & iterative scheme & \cite{Bau}\\
\hline
 Sublevel sets & iterative scheme & \cite{Bau}\\
\hline
 Simplex & iterative scheme & \cite{ParB}\\
\hline
$ x \succcurlyeq 0, x =  \sum_{i=1}^n \lambda_i u_i u_i^T$ & $u= \sum_{i=1}^n (\lambda_i)_+ u_i u_i^T$ & \cite{ParB}\\
\hline
$ x \succcurlyeq 0, tr(x)=1$ & iterative scheme & \cite{ParB}\\
\hline
$ \|\sigma(x)\|_\infty \leq 1, x =  \sum_{i=1}^n \sigma_i u_i v_i^T$ & $u= \sum_{i=1}^n \max(\lambda_i,1) u_i u_i^T$ & \cite{ParB}\\
\hline
\end{tabular}
\end{center}
\end{table}

\vspace{5mm}
\section{A review of OSGA}
In what follows we briefly review the main idea of optimal subgradient 
algorithm proposed by {\sc Neumaier} in \cite{NeuO}. To this end, we 
first consider the convex constrained minimization problem
\lbeq{e.func0}
\begin{array}{ll}
\min &~ f(x)\\
\st  &~ x \in C,
\end{array}
\eeq
where $f:C\to\Rz$ is a convex function defined on a nonempty, closed 
and convex subset $C$ of $\mathcal{V}$. The main objective is to find 
a solution $u\in C$ by using the first-order information, i.e., function values and subgradients.
 
OSGA (see Algorithm 1) is an optimal subgradient algorithm for problem 
(\ref{e.func0}) that constructs a sequence of iterations whose related 
function values converge to the minimum with the optimal complexity. 
Moreover, OSGA 
requires no information regarding global parameters such as Lipschitz 
constants of function values and gradients. The primary 
objective is to monotonically reduce bounds on the error $f(x_b)-\wh f$ 
of function values, where $\wh f$ is the minimum and $x_b$ is the 
best known point.  

OSGA considers the linear relaxations 
\lbeq{e.f1}
f(z)\ge \gamma +\<h,z\> \Forall z\in C,
\eeq
of $f$ at $z$, where $\gamma\in\Rz$ and $h\in \mathcal{V}^*$, and a 
continuously differentiable prox-function $Q:C\to \Rz$ satisfying
\lbeq{e.Qinf}
Q_0:=\inf_{z\in C} Q(z) >0
\eeq
and
\lbeq{e.strc1}
Q(z)\ge Q(x)+\<g_Q(x),z-x\>+\frac{\sigma}{2}\|z-x\|^2 \Forall x,z\in C,
\eeq
where $\sigma = 1$, $g_Q(x)$ denotes the gradient of $Q$ at $x\in C$ and $\|\cdot\|$ 
is a norm defined on $\mathcal{V}$. OSGA solves a sequence of 
minimization problems of the form
\lbeq{e.Eeta}
\begin{array}{ll}
\sup & E_{\gamma,h}(x)\\
\st  &  x \in C,
\end{array}
\eeq
where it is known that the supremum is positive. The function $E_{\gamma,h}: C \rightarrow \Rz$ is defined by
\begin{equation}\label{e.ex}
 E_{\gamma,h}(x):= -\frac{\gamma+\<h,x\>}{Q(x)}.
\end{equation}
If $u=U(\gamma,h)\in C$ is the solution of this problem, then it is
assumed that $e = E(\gamma,h)$ and $u = U(\gamma,h)$ are readily 
computable. 

In \cite{NeuO}, it is shown that OSGA attains the following bound on 
function values
\[
0\le f(x_b) -\wh f\le \eta Q(\wh x).
\]
Hence, by decreasing the error factor $\eta$, the convergence to an $\varepsilon$-minimizer $x_b$ is guaranteed by 
\[
0\le f(x_b) -\wh f\le\eps,
\]
for the accuracy tolerance $\eps>0$.
In \cite{NeuO}, it is shown that the number of iterations to achieve the optimizer is in the order 
$O\left(\varepsilon^{-1/2} \right)$ for smooth $f$ with Lipschitz continuous gradients and in the order $O\left(\varepsilon^{-2} \right)$ for Lipschitz continuous nonsmooth $f$, which is optimal in both cases, cf. {\sc Nemirovsky \& Yudin} \cite{NemY} and {\sc Nesterov} \cite{NesB}. The algorithm does not need to know about the global Lipschitz parameters and has the low memory requirement. Hence if the subproblem (\ref{e.Eeta}) can be solved efficiently, OSGA is appropriate for solving large-scale problems. Numerical results reported by {\sc Ahookhosh} in \cite{Aho} and {\sc Ahookhosh \& Neumaier} in \cite{AhoN1,AhoN2}, for unconstrained problems, and {\sc Ahookhosh \& Neumaier} in \cite{AhoN3,AhoN4}, for constrained problems, show the promising behavior of OSGA for practical problems. In the next section we show that by selecting a suitable prox-function, OSGA's subproblem (\ref{e.Eeta}) can be solved efficiently for structured convex constrained problems. 

\vspace{7mm}
\begin{algorithm}[H] \label{a.osga}
\DontPrintSemicolon 
\KwIn{ $\delta, \alpha_{\max}\in{]0,1[}$,~ 
$0<\kappa'\le\kappa$;~ local parameters: $x_0$,~$\mu \geq 0$,~ 
$f_{\mathrm{target}}$;}
\KwOut{$x_b$,~ $f_{x_b}$;}
\Begin{
    choose an initial best point $x_b$;\;
    compute $f_{x_b}$ and $g_{x_b}$;\;
    \eIf{$f_{x_b} \leq f_{\mathrm{target}}$} {
        stop;\;
    }{
        $h = g_{x_b}-\mu g_Q(x_b)$;~ $\gamma = f_{x_b}-\mu Q(x_b)-\<h,x_b\>$;\;
        $\gamma_b = \gamma - f_{x_b}$;~ $u = U(\gamma_b,h)$;~ $\eta = 
        E(\gamma_b,h)-\mu$;\;
    }
    $\alpha \gets \alpha_{\max}$;\\
    \While {stopping criteria do not hold}{
        $x = x_b+\alpha(u-x_b)$; compute $f_x$ and $g_x$;\;
        $g = g_x-\mu g_Q(x)$;~ $\ol h = h+\alpha(g-h)$;\;
        $\ol\gamma = \gamma+\alpha(f_x - \mu Q(x)-\<g,x\>-\gamma)$;\;
        $x_b' = \argmin_{z \in \{x_b,x\}} f(z, v_z)$;~ $f_{x_b'} = 
        \min \{f_{x_b}, f_x\}$;\;
        $\gamma_b' = \ol\gamma-f_{x_b'}$;~ $u' = U(\gamma_b',\ol h)$;\; 
        $x' = x_b+\alpha(u'-x_b)$; compute $f_{x'}$;\;
        choose $\ol x_b$ in such a way that $f_{\ol x_b}\le \min\{f_{x_b'},f_{x'} \}$;\;
        $\ol\gamma_b = \ol\gamma - f_{\ol x_b}$;~ $\ol u = U(\ol \gamma_b,\ol h)$;~
        $\ol \eta = E(\ol\gamma_b,\ol h)-\mu$; ~$x_b = \ol x_b$; ~ $f_{x_b} = f_{\ol x_b}$;\;
        \eIf {$f_{x_b} \leq f_{\mathrm{target}}$}{
            stop;\;
        }{
            update the parameters $\alpha$, $h$, $\gamma$, $\eta$ and $u$ using UPS;
        }
    }
}
\caption{ {\bf OSGA} (optimal subgradient algorithm)}
\end{algorithm}

\vspace{7mm}
As discussed in \cite{NeuO}, OSGA uses the following scheme for updating the given parameters $\alpha$, $h$, $\gamma$, $\eta$ and $u$:\\

\begin{algorithm}[h] \label{a.par}
\DontPrintSemicolon 
\KwIn{ $\delta$,~ $\alpha_{\max}\in{]0,1[}$,~ $0<\kappa'\le\kappa$, $\alpha$, $\eta$, $\bar{h}$, $\bar{\gamma}$, $\bar{\eta}$, $\bar{u}$;}
\KwOut{$\alpha$,~ $h$,~ $\gamma$,~ $\eta$,~ $u$;}
\Begin{
    $R \gets \left(\eta-\ol \eta)/(\delta\alpha \eta \right)$;\;
    \eIf{$R<1$} {
        $h \gets \ol h$;\;
    }{
        $\ol\alpha \leftarrow \min(\alpha e^{\kappa' (R-1)},\alpha_{\max})$;
    }
    $\alpha \gets \ol \alpha$;\;
    \If{$\ol \eta<\eta$}{
        $h \gets \ol h$;~ $\gamma \gets \ol \gamma$;~ $\eta \gets \ol \eta$;~ $u \gets \ol u$;\;
    }
}
\caption{ {\bf PUS} (parameters updating scheme)}
\label{algo:max}
\end{algorithm}

\section{Structured convex constrained problems in simple domains} \label{s.stru}

In this paper we consider the convex constrained optimization problem
\lbeq{e.func1}
\begin{array}{ll}
\min &~ f(\mathcal{A}x)\\
\st  &~ x \in C,
\end{array}
\eeq
where $f: C \to \Rz$ is convex and lower semicontinuous, 
$\mathcal{A}: \mathbb{R}^n \rightarrow \mathbb{R}^m$ is a linear 
operator, and $C$ is a simple convex domain. We call problem 
(\ref{e.func1}) a {\bf simple domain} problem. 
This problem appears in many 
applications such as signal and image processing, machine learning, 
statistics, and inverse problem. 

\begin{exam} {\sc (Image restoration)}
The process of reconstructing or estimating a true image from a degraded observation is known as the image restoration, also called deblurring or deconvolution. Image restoration is addressed by solving a constraint satisfaction problem of the form
\[
\mathcal{A}x = b, ~~ x \in C,
\]
where $C$ a convex domain $C$ that is commonly a box or the nonnegativity constraint. This is an ill-posed problem, see {\sc Neumaier} \cite{NeuI}, and normally handled by the regularized least-squares problem
\begin{equation} \label{e.l22itvr}
\begin{array}{ll}
\min & \frac{1}{2} \|\mathcal{A}x-b\|_2^2 + \lambda \varphi(x)\\
\mathrm{s.t.} & x \in C
\end{array}
\end{equation}
or the regularized $l_1$ problem
\begin{equation} \label{e.l1itvr}
\begin{array}{ll}
\min & \|\mathcal{A}x-b\|_1 + \lambda \varphi(x)\\
\mathrm{s.t.} & x \in C,
\end{array}
\end{equation}
where $\varphi: C \rightarrow \mathbb{R}$ is a convex regularization function such as $\|\cdot\|_2^2$, $\|\cdot\|_1$, $\|\cdot\|_{ITV}$, and $\|\cdot\|_{ATV}$. The regularizers $\|\cdot\|_{ITV}$ and $\|\cdot\|_{ATV}$ are respectively called isotropic and anisotropic total variation, see, for example, \cite{ChaCCNP}, where they are defined by
\[
    \bary{lll}
    \|x\|_{ITV} &=& \sum_i^{m-1} \sum_j^{n-1} \sqrt{(x_{i+1,j} - x_{i,j})^2+(x_{i,j+1} - x_{i,j})^2 }\\
            &+& \sum_i^{m-1} |X_{i+1,n} - Xx_{i,n}| + \sum_i^{n-1} |x_{m,j+1} - x_{m,j}|
    \eary
\]
and
    \[
    \bary{lll}
    \|x\|_{ATV} &=& \sum_i^{m-1} \sum_j^{n-1} \{|x_{i+1,j} - x_{i,j}| + |x_{i,j+1} - x_{i,j}| \}\\
            &+& \sum_i^{m-1} |x_{i+1,n} - x_{i,n}| + \sum_i^{n-1} |x_{m,j+1} - x_{m,j}|,
    \eary
\]
for $x \in \Rz^{m\times n}$. 
\end{exam}

\begin{exam} {\sc(Basis pursuit problem)} 
Let $\mathcal{A}: \mathbb{R}^n\rightarrow\mathbb{R}^m$ be a linear operator with $m < n$ and $y \in \mathbb{R}^m$. The basis pursuit problem is the constrained minimization problem
\lbeq{e.bpp}
\bary{ll}
\min & \|x\|_1 \\
\st  & \mathcal{A}x = y,
\eary
\eeq
which determines an $l_1$-minimal solution $\wh x$ of the undetermined linear system $\mathcal{A}x = y$. This problem appears in many applications such signal and image processing and compressed sensing, see \cite{BerF1,BerF2,CheDS, Don, WenYGZ, Yin, YinOGD} and references therein. 
\end{exam}

According to the features of objective functions,  (\ref{e.l22itvr}) can be solved by Nesterov-type optimal methods, however, (\ref{e.l1itvr}) and (\ref{e.bpp}) cannot be solved by Nesterov-type optimal methods. Since OSGA only needs first-order information, it can deal with all of these problems without considering the structure of problems. In the remainder of this section, we establish how OSGA can be used to efficiently solve the problem (\ref{e.func1}). Since the underlying problem (\ref{e.func1}) is a special case of the problem (\ref{e.func0}) considered in \cite{NeuO}, the complexity of OSGA remains valid for both smooth and nonsmooth problems. 

The quadratic function
\lbeq{e.prox}
Q(z):= \half \|z\|_2^2 + Q_0,
\eeq
is a prox-function, see e.g. \cite{Aho}. We now show that the solution of OSGA's subproblem (\ref{e.Eeta}) can be found either in a closed form or by a simple iterative scheme. In particular, we address some convex domains that a 
closed form solution for associated OSGA's subproblem (\ref{e.Eeta}) 
can be found.

The next result shows that the solution of the auxiliary subproblem 
(\ref{e.Eeta}) is given by the orthogonal projection (\ref{e.pro}) of 
$y := e^{-1} h$ on the domain $C$ followed by solving a one-dimensional 
nonlinear equation to determine $e$.

\begin{thm} \label{t.sol}
Let $u$ be a minimizer of (\ref{e.Eeta}) and also let $e = E_{\gamma, h} (u)>0$. Then 
\[
u = \wh u(e):= P_C(y),~~ y := -e^{-1} h,
\]
where, $e$ is a solution of the univariate equation 
\[
\varphi(e) = 0
\]
with
\lbeq{e.none}
\varphi(e) := e \left(\frac{1}{2}\|\wh u(e)\|_2^2 + Q_0 \right) + \gamma + \langle h,\wh u(e) \rangle.
\eeq
\end{thm}

\begin{proof}
From Proposition 5.1 in \cite{NeuO}, at the minimizer $u$, we obtain
\lbeq{e.EQ2}
e Q(u) = -\gamma - 
\langle h,u \rangle
\eeq
and
\lbeq{e.ineq1}
\langle e u + h, z-u \rangle \geq 0~~~ 
\mathrm{for~all}~ z \in C.
\eeq
By setting $z = u$ in this variational inequality, it follows that $u$ is a solution of the minimization problem
\[
\inf_{z \in C}~ \langle e u + h, z-u \rangle.
\]
The first-order optimality condition for this problem is
\lbeq{e.OP}
0 \in eu + h + N_C(u),
\eeq
where
\[
N_C(u) := \left \{ p \in \mathcal{V} \;|\; \forall y \in C, \langle p, u - y \rangle \geq 0 \right\}
\]
denotes the normal cone to $C$ at $u$. Since $e>0$, $u$ satisfies
\[
u = \argmin_{z\in C} \frac{1}{2}\|ez+h\|_2^2 = \argmin_{z\in C} \frac{1}{2}\|z - y\|_2^2 = P_C(y)= \wh u(e),
\]
where $y = -e^{-1} h$ giving the result.
\end{proof}

Theorem \ref{t.sol} gives a way to compute a solution of OSGA's subproblem (\ref{e.Eeta}) involving a projection on the domain $C$ and solving the one-dimensional nonlinear equation. This equation can be solved exactly for some projection operators, see Table 2. However, one can solve this nonlinear equation approximately using zero finding schemes, see e.g. Chapter 5 of \cite{NeuB}. We apply the results of Theorem \ref{t.sol} in the next scheme to solve OSGA's subproblem (\ref{e.Eeta}):

\vspace{5mm} 
\begin{algorithm}[H] \label{a.oss}
\DontPrintSemicolon 
\KwIn{$Q_0$,~ $\gamma$,~ $h$. a program for evaluating $\varphi(e)$ defined in (\ref{e.none});}
\KwOut{$u$,~ $e$;}
\Begin{
solve the nonlinear equation $\varphi(e)=0$ either in a closed form or approximately by a root finding solver;\;
set $u = \wh u(e)$.
}
\caption{ {\bf OSS} (OSGA's subproblem solver)}
\end{algorithm}

\vspace{5mm}
To implement Algorithm \ref{a.oss} (OSS), we first need to solve the projection problem (\ref{e.pro}) effectively, see Table \ref{t.proj}. If one solves the equation $\varphi(e)=0$ approximately, and an initial interval $[a,b]$ is available such that $\varphi(a)\varphi(b)<0$, then a solution can be computed to $\varepsilon$-accuracy using the bisection scheme in $O(\log_2((b-a)/\varepsilon))$ iterations, see, for example, \cite{NeuB}. However, it is preferable to use a more sophisticated zero finder like the secant bisection scheme (Algorithm 5.2.6, \cite{NeuB}). If an interval $[a,b]$ with sign change is available\footnote{
Without a sign change, $\mathtt{fzero}$ is unreliable; it fails on the 
simple quadratic $x^2-0.0001=0$ with starting point $0.2$.
}, 
one can also use
MATLAB's $\mathtt{fzero}$ function combining the bisection scheme, the inverse quadratic interpolation, and the secant method. 

In the following we investigate special domains $C$, where the nonlinear equation $\varphi(e)=0$ can be solved explicitly, see Table 2. 

\begin{table}\label{t.expl}
\caption{List of domains $C$ where $\varphi(e)=0$ can be solved explicitly}
\begin{center}\footnotesize
\renewcommand{\arraystretch}{1.3}
\begin{tabular}{|l|l|}\hline
\multicolumn{1}{|l|}{{\bf defining constraint $c(x)$}} & \multicolumn{1}{l|}{{\bf solution}}\\ 
\hline
$Ax=b$ & Proposition \ref{p.aff} \\
\hline
$\langle a,x \rangle = b$ & Corollary \ref{c.hyp}\\
\hline
$\langle a,x \rangle \leq b$ & Proposition \ref{p.half}\\
\hline
$x \geq 0$ & Proposition \ref{p.nego}\\
\hline
$ \|x\|_2 \leq \xi$ & Proposition \ref{p.eball}\\
\hline
\end{tabular}
\end{center}
\end{table}

\begin{proposition} \label{p.aff}
If $C = \{ x \in \mathcal{V} ~\mid~ Ax=b\}$ is an affine set, then the subproblem (\ref{e.Eeta}) is solved by $u = P_C(-e^{-1} h)$, where
\lbeq{e.affproj}
P_C(y) = y - A^\dagger (Ay-b).
\eeq
and 
\lbeq{e.aff}
e = \frac{-\beta_2 + \sqrt{\beta_2^2-4\beta_1 \beta_3}}{2\beta_1},
\eeq
with
\lbeq{e.affpar}
\beta_1 := \frac{1}{2} \|A^\dagger b\|_2^2 +Q_0,~~ \beta_2 := \langle A^\dagger (Ah),A^\dagger b \rangle+\gamma,~~
\beta_3 := \frac{1}{2} \|A^\dagger (Ah)\|_2^2 + \frac{1}{2} \|h\|_2^2. 
\eeq
\end{proposition}

\begin{proof}
The projection operator on $C$ is given by (\ref{e.affproj}). This and $y = -e^{-1} h$ give
\[
P_C(-e^{-1}h) = -e^{-1} (A^\dagger (Ah + eb) - h).
\]
This, together with (\ref{e.EQ2}), yields
\[
\begin{split}
e Q(u) + \gamma + 
\langle h,u \rangle &= e \left(\frac{1}{2}(\|P_C(-e^{-1}h)\|_2^2) + Q_0 \right) + \gamma + 
\langle h,P_C(-e^{-1}h) \rangle\\
&= \frac{1}{2} \|A^\dagger (Ah + eb)\|_2^2 + \frac{1}{2} \|h\|_2^2 - \langle A^\dagger (Ah + eb),h \rangle + Q_0 e^2 \\
&~~~+ \gamma e + \langle A^\dagger (Ah + eb)-h,h \rangle\\
& = \left( \frac{1}{2} \|A^\dagger b\|_2^2 +Q_0 \right) ~e^2 + (\langle A^\dagger (Ah),A^\dagger b \rangle+\gamma)~e \\
&~~~ +\frac{1}{2} \|A^\dagger (Ah)\|_2^2 + \frac{1}{2} \|h\|_2^2\\
& = \beta_1 e^2 + \beta_2 e + \beta_3 = 0,
\end{split}
\]
where $\beta_1$, $\beta_2$, and $\beta_3$ are defined in (\ref{e.affpar}). Since the subproblem (\ref{e.Eeta}) is the maximization, the bigger root of this equation is selected, which is given by (\ref{e.aff}).
\end{proof}

\begin{corollary} \label{c.hyp}
If $C = \{ x \in \mathcal{V} ~\mid~ a^Tx=b\}$ is a hyperplane, then the subproblem (\ref{e.Eeta}) is solved by $u = P_C(-e^{-1} h)$, where
\textbf{\lbeq{e.phyper}
P_C(y) = y - \left( \frac{\langle a,y \rangle - b}{\|a\|_2^2} \right) a,
\eeq} 
and $e$ is given by (\ref{e.aff}) with
\lbeq{e.hyper}
\beta_1 := \frac{b}{2 \|a\|_2^2} + Q_0, ~~ \beta_2 := \frac{b\langle a,h \rangle}{\|a\|_2^2} + \gamma, ~~ \beta_3 := \frac{1}{2} \frac{\langle a,h \rangle^2}{\|a\|_2^2} - \frac{1}{2} \|h\|_2^2. 
\eeq
\end{corollary}

\begin{proof}
Since the hyperplane $C = \{ x \in \mathcal{V} ~\mid~ a^Tx=b\}$ is an affine set, this is a special case of Proposition \ref{p.aff}.
\end{proof}

\begin{proposition} \label{p.half}
If $C = \{ x \in \mathcal{V} ~\mid~ \langle a,x \rangle \leq b \}$ is a halfspace, then the subproblem (\ref{e.Eeta}) is solved by $u = P_C(-e^{-1} h)$, where
\lbeq{e.halfproj}
P_C(y) =  y - \frac{(\langle a,y \rangle - b)_+}{\|a\|_2^2}~ a
\eeq
and $e$ is given by (\ref{e.aff}) with
\lbeq{e.halfpar2}
\begin{array}{l}
\beta_1 := Q_0,~~ \beta_2 := \gamma,~~ \beta_3 := -\frac{1}{2}\|h\|_2^2,
\end{array}
\eeq
say $e_1$, and with $\beta_1$, $\beta_2$, and $\beta_3$ is given in (\ref{e.hyper}), say $e_2$. If $\langle a,h \rangle \geq e_1^{-1}b$ and $\langle a,h \rangle \geq e_2^{-1}b$, then $e = e_1$. If $\langle a,h \rangle \leq e_1^{-1}b$ and $\langle a,h \rangle < e_2^{-1}b$, then $e = e_2$. If $\langle a,h \rangle \geq e_1^{-1}b$ and $\langle a,h \rangle < e_2^{-1}b$, then $e = \max \{e_1, e_2\}$.
\end{proposition}

\begin{proof}
The projection operator on $C$ is given by (\ref{e.halfproj}). This gives
\lbeq{e.halfpro}
P_C(-e^{-1}h) = -e^{-1} \left(h + \frac{(\langle a,h \rangle + eb)_-}{\|a\|_2^2}~ a \right).
\eeq
If $\langle a,h \rangle \geq -eb$, we obtain
\[
P_C(-e^{-1}h) = -e^{-1}h,
\] 
leading to
\[
\begin{split}
e Q(P_C(-e^{-1}h)) + \gamma + 
\langle h,P_C(-e^{-1}h) \rangle &= \frac{1}{2} e^{-1} \|h\|_2^2 + Q_0 e + \gamma - e^{-1}
\|h\|_2^2\\
& = Q_0 e^2 + \gamma e - \frac{1}{2}\|h\|_2^2  = \beta_1 e^2 + \beta_2 e + \beta_3 = 0,
\end{split}
\]
where $\beta_1 := Q_0$, $\beta_2 := \gamma$, and $\beta_3 := -\frac{1}{2}\|h\|_2^2$. This identity leads to a solution of the form (\ref{e.aff}), say $e_1$. If $\langle a,h \rangle < -eb$, (\ref{e.phyper}) is valid and $e$ is computed by (\ref{e.aff}) where $\beta_1$, $\beta_2$, and $\beta_3$ is defined in (\ref{e.hyper}), say $e_2$. After computing $e_1$ and $e_2$, we check whether the inequalities $\langle a,h \rangle \geq -e_1 b$ and $\langle a,h \rangle < -e_2 b$ are satisfied. Since the subproblem (\ref{e.Eeta}) has a solution, at least one of the conditions has to satisfied. If one of them is satisfied, the corresponding $e$ and (\ref{e.halfpro}) give the solution. If both of them hold, we consider the solution with bigger $e$. 
\end{proof}

\begin{proposition} \label{p.nego}
If $C = \{ x \in \mathbb{R}^n ~\mid~ x_i \geq 0~~~ i=1, \cdots, n\}$ is the nonnegative orthant, then the subproblem (\ref{e.Eeta}) is solved by $u = P_C(-e^{-1} h)$, where 
\lbeq{e.negproj}
P_C(y) = (y)_+
\eeq
and $e$ is given by (\ref{e.aff}) with
\lbeq{e.negopar}
\beta_1 := Q_0,~~ \beta_2 := \gamma, ~~ \beta_3 := \frac{1}{2}\|(h)_-\|_2^2 -\langle h,(h)_- \rangle. 
\eeq
\end{proposition}

\begin{proof}
The projection operator on $C$ is given by (\ref{e.negproj}) leading to
\[
P_C(-e^{-1}h) = - e^{-1} (h)_-.
\]
This and (\ref{e.EQ2}) imply
\[
\begin{split}
e Q(P_C(-e^{-1}h)) + \gamma + 
\langle h,P_C(-e^{-1}h) \rangle &= \frac{1}{2}e^{-1} \|(h)_-\|_2^2 + Q_0 e + \gamma - e^{-1}
\langle h,(h)_- \rangle\\
& = Q_0 e^2 + \gamma e + \frac{1}{2}\|(h)_-\|_2^2 -\langle h,(h)_- \rangle \\
& = \beta_1 e^2 + \beta_2 e + \beta_3 = 0,
\end{split}
\]
where $\beta_1$, $\beta_2$, and $\beta_3$ are defined in (\ref{e.negopar}), giving the result.
\end{proof}

\begin{proposition} \label{p.eball}
Let $C = \{ x \in \mathbb{R}^n ~\mid~\|x\|_2 \leq \xi \}$ be the Euclidean ball. Then
\lbeq{e.eballproj}
P_C(y) = \left\{
\begin{array}{ll}
\xi y / \|y\|_2 & ~~ \|y\|_2 > \xi,\\
y           & ~~ \|y\|_2 \leq \xi,
\end{array}
\right.
\eeq
If $\|e^{-1} h\|_2 \leq \xi$ where $e$ is given by (\ref{e.aff}) with
\lbeq{e.eballpar}
\beta_1 := Q_0,~~ \beta_2 := \gamma, ~~ \beta_3 := -\frac{1}{2} \|h\|_2^2, 
\eeq
then $u = -e^{-1} h$; otherwise, the solution of OSGA's subproblem (\ref{e.Eeta}) is given by 
\[
u = -\frac{\xi}{\|h\|_2} h,~~~e = -\frac{2(\gamma + \xi \|h\|_2)}{\xi^2 + 2 Q_0}.
\]
\end{proposition}

\begin{proof}
The projection operator on $C$ is given by (\ref{e.eballproj}), leading to
\[
P_C(-e^{-1}h) = \left\{
\begin{array}{ll}
-\xi h / \|h\|_2 & ~~ \|h\|_2 > e\xi,\\
-e^{-1} h    & ~~ \|h\|_2 \leq e\xi.
\end{array}
\right.
\]

We first assume that $\|h\|_2 \leq e\xi$ implying $P_C(-e^{-1}h) = -e^{-1}h$. Substituting this into (\ref{e.EQ2}) yields
\[
\begin{split}
e Q(P_C(-e^{-1}h)) + \gamma + 
\langle h,P_C(-e^{-1}h) \rangle &= \frac{1}{2} e^{-1} \|h\|_2^2 + Q_0 e + \gamma - e^{-1}
\|h\|_2^2\\
& = Q_0 e^2 + \gamma e - \frac{1}{2}\|h\|_2^2  = \beta_1 e^2 + \beta_2 e + \beta_3 = 0,
\end{split}
\]
where $\beta_1 := Q_0$, $\beta_2 := \gamma$, and $\beta_3 := -\frac{1}{2}\|h\|_2^2$. Hence $e$ is given by (\ref{e.aff}). If this $e$ satisfies $\|h\|_2 \leq e\xi$, then $u = -e^{-1}h$. Otherwise, we assume that $\|h\|_2 > e\xi$. Substituting $P_C(-e^{-1}h) = -\xi h / \|h\|_2$ into (\ref{e.EQ2}) yields
\[
e \left( \frac{1}{2} \xi^2 + Q_0 \right) + \gamma - \xi \|h\|_2=0,
\]
implying
\[
e = -\frac{2(\gamma + \xi \|h\|_2)}{\xi^2 + 2 Q_0}
\]
and $u = -\xi h / \|h\|_2$. This completes the proof.
\end{proof}

To solve bound-constrained problems with OSGA, we developed and algorithm that can find the global solution of the subproblem (\ref{e.Eeta}) by solving a sequence of one-dimensional rational optimization problems, see Algorithm 3 in \cite{AhoN3}. Notice that the constraint $C := \{x \in \mathcal{V} \mid \|x\|_\infty \leq \xi\}$ is a special case of bound-constrained problem with $\underline{x} = -\xi {\bf 1}$ and $\overline{x} = \xi {\bf 1}$ where ${\bf 1}$ is a $n$-dimensional vector with all elements equal to unity.

\section{Solving structured problems with a functional constraint}
In this subsection we consider the structured convex constrained 
problem 
\lbeq{e.func10}
\begin{array}{ll}
\min &~ f(\mathcal{A}x)\\
\st  &~ \phi(x) \leq \xi,
\end{array}
\eeq
where $\phi:C \rightarrow \overline{\mathbb{R}}$ is a simple smooth or 
nonsmooth, real-valued, and convex loss function, and $\xi$ is a real 
constant. We call the problem (\ref{e.func10}) a {\bf functional 
constraint} problem. While it the special case of 
(\ref{e.func1}) with
\[
C:=\{ x \in \mathcal{V} ~|~ \phi(x)\le \xi\},
\]
one can solve OSGA's subproblem (\ref{e.Eeta}) directly by using the KKT optimality conditions, especially when no  efficient method for finding the projection on $C$ is known. Indeed, if a nonsmooth problem can be reformulated in the form (\ref{e.func1}) with a smooth $f$ and a nonsmooth $\phi$, then OSGA can solve this nonsmooth problem with the complexity of the order $O(\varepsilon^{-1/2})$, which is optimal for smooth problems.

\begin{exam}{\sc (Linear inverse problem)} 
Let $\mathcal{A}: \mathbb{R}^n\rightarrow\mathbb{R}^m$ be an 
ill-conditioned or singular linear operator and $y \in \mathbb{R}^m$ 
be a vector of observations. The linear inverse problem is the quest of
finding $x \in \mathbb{R}^n$ such that
\begin{equation} \label{e.lin}
y = \mathcal{A}x + \nu,
\end{equation}
with unknown but small additive noise $\nu \in \mathbb{R}^m$.
The problem is solvable if one knows additional qualitative information 
about $x$. This qualitative information is encoded in a constraint on
$x$, under which the Euclidean norm of $\nu$ is minimized.  
Constrained optimization problems resulting from two typical 
qualitative constraints are
\lbeq{e.las2}
\bary{ll}
\min & \frac{1}{2} \|y-\mathcal{A}x\|_2^2 \\
\st  & \|x\|_2 \leq \xi,
\eary
\eeq
\lbeq{e.las3}
\bary{ll}
\min & \frac{1}{2} \|y-\mathcal{A}x\|_2^2 \\
\st  & \|x\|_{1,2} \leq \xi,
\eary
\eeq
in which $\xi$ is a nonnegative real constant. 
This problem often occurs in applied sciences and engineering, see \cite{KimKY,Tib}.
\end{exam}

In the reminder of this section we assume that the functional constraint
satisfies the {\bf Cottle constraint qualification} \cite{BagKM}\\
{\bf (H1)} For all $x \in C$, either 
$\phi(x)<0$ or $0 \not\in \partial_\phi(x)$.
\\
We also need the following result.

\begin{proposition} \label{p.subd} 
(see, e.g., \cite{AhoN2})
Let $\phi:\mathcal{V} \rightarrow \mathbb{R},~ \phi(x) = \|x\|$. Then the subdifferential of $\phi$ is 
\[
\partial \phi (x)= \left\{
\begin{array}{ll}
\{g ~|~ \|g\|_* \leq 1\} & ~~ \mathrm{if} ~ x=0,\\
\{g ~|~ \|g\|_* = 1,~ \langle g,x \rangle = \|x\| \} & ~~ \mathrm{if} ~ x \neq 0.
\end{array}
\right.
\]
Moreover, if $\|\cdot\|$ is self-dual, then
\[
\partial \phi (x) = \left\{
\begin{array}{ll}
\{g ~|~ \|g\|_* \leq 1\} & ~~ \mathrm{if} ~ x=0,\\
 x/\|x\|  & ~~ \mathrm{if} ~ x \neq 0.
\end{array}
\right.
\]
\end{proposition}
The next result gives the optimality conditions for solving the problem (\ref{e.func1}).

\begin{thm} \label{t.opt1}
Let (H1) satisfies for the problem (\ref{e.func10}). Then, for a real constant $\xi$, the solution $u$ of OSGA's subproblem 
\[
\begin{array}{ll}
\D \min &~ \D\frac{-\gamma-\langle h,x \rangle}{Q(x)}\\
\st  &~ \phi(x) \leq \xi,
\end{array}
\]
satisfies either
\lbeq{e.opt2}
u = -e^{-1} h,~~~\mu = 0,~~~~\phi(u) < \xi
\eeq
or
\lbeq{e.opt1}
\frac{1}{\mu}\frac{-e u - h}{Q(u)} \in \partial\phi(u),~~~
\mu > 0,~~~ \phi(u)= \xi,
\eeq
where $e := -(\gamma+\langle h,u \rangle)/ Q(u)$.
\end{thm}

\begin{proof}
Let's define the function 
\[
E_{\gamma,h}:C \rightarrow \mathbb{R},~~~ E_{\gamma,h}(x) := -\frac{\gamma+\langle h,x \rangle}{Q(x)}.
\]
Since this function is differentiable, by differentiating both sides of the equality $E_{\gamma,h}(x) Q(x) = -\gamma -\langle h,x \rangle$ with respect to $x$, we obtain
\lbeq{e.dif}
\partial E_{\gamma,h}(x) = \left\{ \frac{-E_{\gamma,h}(x) x - h}{Q(x)} \right\}.
\eeq
In view of the KKT optimality conditions for inequality constrained nonsmooth problems, see \cite{BagKM}, we have the optimality condition 
\lbeq{e.kkt}
\left\{
\begin{array}{l}
0 \in \partial E_{\gamma,h}(u) + \mu \partial\phi(u),\\
\phi(u) \leq \xi,\\
\mu \geq 0,\\
\mu (\phi(u)-\xi) = 0,
\end{array}
\right.
\eeq
for (\ref{e.func1}). Now, by substituting (\ref{e.dif}) into (\ref{e.kkt}), setting $e := -(\gamma+\langle h,u \rangle)/ Q(u)$, and distinguishing between $\mu=0$ and $\mu>0$, we obtain either (\ref{e.opt2}) or (\ref{e.opt1}).
\end{proof}

Theorem \ref{t.opt1} gives the optimality conditions for general function $\phi$, however, in view of Theorem \ref{t.sol}, it is especially useful when the projection in $C=\{x \mid \phi(x)\leq \xi\}$ is not efficiently available.
In the remainder of this subsection, we derive the solution of OSGA's subproblem (\ref{e.Eeta}) for some  $\phi$ such as $\|\cdot\|_2$ and $\|\cdot\|_{1,2}$ that appear in many applications. We already solve OSGA's subproblem (\ref{e.Eeta}) with the constraint $C=\{x \mid \|x\|_2 \leq \xi\}$ in Proposition \ref{p.eball}, but to show how to apply Theorem \ref{t.opt1} we study it in the next result.

\begin{prop} \label{p.ssl2}
Let $\mathcal{V}$ be a real finite-dimensional Hilbert space with the induced norm $\phi(\cdot) = \|\cdot\|_2$. Then OSGA's subproblem (\ref{e.Eeta}) is solved by
\[
u = -e^{-1} h, ~~ e = \frac{-\beta_2 + \sqrt{\beta_2^2 -4\beta_1 \beta_3}}{-2\beta_1}, ~~\mu=0,
\]
where 
\[
\beta_1 := Q_0,~~ \beta_2 := \gamma, ~~\beta_3 := \frac{1}{2} \|h\|_2^2,
\]
if $\phi(u) < \xi$; Otherwise it is solved by
\[
u = \frac{\xi}{\|h\|_2} h, ~~ e = -\frac{2 \|h\|_2 (\gamma \|h\|_2 + \xi \|h\|_2)}{\xi^2 \|h\|_2^2 + 2Q_0\|h\|_2^2}., ~~ \mu = \frac{2 (\|h\|_2 + e\xi)\|h\|_2^2}{\|h\|_2^2+2Q_0 \|h\|_2^2}.
\]
\end{prop}

\begin{proof}
Since $\|\cdot\|_2$ is self-dual, Proposition \ref{p.subd} implies
\[
\partial \phi(u) = \left\{
\begin{array}{ll}
\{g \in \mathcal{V}^* ~|~ \|g\|_2 \leq 1 \} & ~~\mathrm{if} ~ u=0,\\
\frac{u}{\|u\|_2}         & ~~\mathrm{if} ~ u \neq 0.
\end{array}
\right.
\]
As $u =0$ is not useful in our optimization setting, we seek only $u \neq 0$. We now apply Theorem \ref{t.opt1} leading to two cases: (i) (\ref{e.opt2}) holds; (ii) (\ref{e.opt1}) holds. 

Case (i). The condition (\ref{e.opt2}) holds. Then we have $u = -e^{-1} h$. By substituting this into the identity $E_{\gamma, h} (u) = e$, we get
\[
e = -\frac{\gamma - \|h\|_2^2 ~e^{-1}}{\frac{1}{2} \|h\|_2^2 ~e^{-2} + Q_0},
\]
implying 
\[
Q_0 e^2 + \gamma e - \frac{1}{2} \|h\|_2^2 = 0.
\]
By using the bigger root of this equation, we have
\[
e = \frac{-\beta_2 + \sqrt{\beta_2^2 -4\beta_1 \beta_3}}{-2 \beta_1},
\]
where $\beta_1=Q_0$, $\beta_2 = \gamma$, and $\beta_3= \frac{1}{2} \|h\|^2$.

Case (ii). The condition (\ref{e.opt1}) holds. Then we have
\[
\frac{-eu - h}{\frac{1}{2} \|u\|_2^2 + Q_0} = -\mu \frac{ u}{\|u\|_2},
\]
giving
\[
(-e u - h)\|u\|_2 + \mu \left( \frac{1}{2} \|u\|_2^2 + Q_0 \right) u = 0,
\]
leading to
\lbeq{e.sxh}
(-e \|u\|_2 + \frac{1}{2} \mu \|u\|_2^2 + \mu Q_0) u = \|u\|_2 h.
\eeq
This implies that there exist $\lambda$ such that $u = \lambda h$. By substituting this into $\phi(u) = \|u\|_2 = \xi$ we get
\[
\lambda = \frac{\xi}{\|h\|_2}.
\]
Now, substituting $u$ into (\ref{e.sxh}), we obtain
\lbeq{e.smu}
\mu = \frac{2 (\|h\|_2 + e\xi)\|h\|_2^2}{\|h\|_2^2+2Q_0 \|h\|_2^2}.
\eeq
It follows from $E_{\gamma, h} (u) = e$ that
\[
e = -\frac{2 \|h\|_2 (\gamma \|h\|_2 + \xi \|h\|_2)}{\xi^2 \|h\|_2^2 + 2Q_0\|h\|_2^2}.
\]
This gives the result.
\end{proof}

In 2004, {\sc Yuan} and {\sc Lin} in \cite{YuaL} proposed an 
interesting regularizer called grouped LASSO for the linear regression. 
Later {\sc Kim} et al. in \cite{KimKY} proposed a constrained ridge 
regression model using the constraint
\[
\|x\|_{1,2} \leq \xi,
\]
where
\[
\|x\|_{1,2} := \sum_{i=1}^m \|x_{g_i}\|_2,
\]
where $x = (x_{g_1}, \cdots, x_{g_m})$ and $\|x\|_{1,2}$ is a so-called the $l_{1,2}$ group norm. We consider this constraint in the next result.

\begin{prop} \label{p.sl12}
Let $\mathcal{V}$ be a real finite-dimensional vector space with the induced norm $\phi(\cdot) = \|\cdot\|_{1,2}$. Then OSGA's subproblem (\ref{e.Eeta}) is solved by
\[
u_{g_i} = -e^{-1} h_{g_i} ~~~ \mathrm{for~all}~ i = 1, \cdots, m,
\]
and
\[
e = \frac{-\beta_2 + \sqrt{\beta_2^2 -4\beta_1 \beta_3}}{-2\beta_1}, ~~\mu=0,
\]
where 
\[
\beta_1 :=Q_0,~~ \beta_2 := \gamma,~~ \beta_2 := \frac{1}{2} \|h\|_2^2 - \sum_{i=1}^m \|h_{g_i}\|_2^2,
\]
if $\phi(u) < \xi$; Otherwise it is solved by
\[
u_i = \rho_i h_{g_i}, ~~\rho_i = \frac{\|h_{g_i}\|_2 - \mu\left(\frac{1}{2} \xi^2 + Q_0 \right)}{e \|h_{g_i}\|_2} ~~~ \mathrm{for~all}~ i = 1, \cdots, m,
\]
and
\[
e = -\frac{\gamma+\langle h,u \rangle}{\frac{1}{2} \xi^2 + Q_0} = -\frac{2(\gamma + \sum_{i=1}^n \tau_i^2 \|h_{g_i}\|_2^2)}{\sum_{i=1}^n \tau_i^2 \|h_{g_i}\|_2^2 + 2Q_0}, ~~ \mu = \frac{2(\sum_{i=1}^m \|h_{g_i}\|_2 +e \xi)}{m(\sum_{i=1}^m \tau_i^2 \| h_{g_i} \|_2^2 + 2Q_0)}.
\]
\end{prop}

\begin{proof}
Similar to Proposition \ref{p.ssl2}, we consider $u \neq 0$. In view of Proposition (\ref{p.subd}), we get
\[
\partial \phi(u_{g_i}) = \left\{ \frac{u_{g_i}}{\|u_{g_i}\|_2}  \right\} ~~~ \Forall i=1, \cdots, m,          
\]
leading to 
\[
\partial \phi(u) = \left\{ \left( \frac{u_{g_1}}{\|u_{g_1}\|_2}, \cdots, \frac{u_{g_m}}{\|u_{g_m}\|_2} \right) \right\}.          
\]
We now apply Theorem \ref{t.opt1} leading to two cases: (i) (\ref{e.opt2}) holds; (ii) (\ref{e.opt1}) holds.  

Case (i). The condition (\ref{e.opt2}) holds. Then we have $u_{g_i} = -e^{-1} h_{g_i}$ for $i = 1, \cdots, n$. By substituting $u=(u_{g_1}, \cdots,u_{g_n})$ into the identity $E_{\gamma, h} (u) = e$, we get
\[
e = \frac{-\gamma + \sum_{i=1}^m \|h_{g_i}\|_2^2 ~ e^{-1}}{\frac{1}{2} \|h\|_2^2 ~ e^{-2} + Q_0},
\]
implying 
\[
Q_0 e^2 + \gamma e + \frac{1}{2} \|h\|_2^2 - \sum_{i=1}^m \|h_{g_i}\|_2^2= 0.
\]
By using the bigger root of this equation, we get
\[
e = \frac{-\beta_2 + \sqrt{\beta_2^2 -4\beta_1 \beta_3}}{-2\beta_1},
\]
where $\beta_1 :=Q_0$, $\beta_2 := \gamma$, and $\beta_3 := \frac{1}{2} \|h\|_2^2 - \sum_{i=1}^m \|h_{g_i}\|_2^2$.

Case (ii). The condition (\ref{e.opt1}) holds. Then we have
\[
\frac{-e u_{g_i} - h_{g_i}}{\frac{1}{2} \|u\|_2^2 + Q_0} = -\mu \frac{u_{g_i}}{\| u_{g_i} \|_2} ~~~ \Forall i = 1, \cdots, m.
\]
Since $\phi(u) = \|u\| = \xi$, we equivalently get
\[
\left(\frac{1}{2} \|u\|_2^2 + Q_0 \right) \left( -\frac{e}{\frac{1}{2} \|u\|_2^2 + Q_0} + \frac{\mu}{\| u_{g_i} \|_2} \right) u_{g_i} = h_{g_i}
\]
implying $u_{g_i} = \tau_i h_{g_i}$. If $h_{g_i} = 0$, then $u_{g_i} = 0$. Now let $h_{g_i} \neq 0$. Substituting $u_{g_i} = \tau_i h_{g_i}$ into the previous identity, it follows that
\[
\left(\frac{1}{2} \sum_{i=1}^m \tau_i^2 \| h_{g_i} \|_2^2 + Q_0 \right) \left( -\frac{e}{\frac{1}{2} \sum_{i=1}^m \tau_i^2 \| h_{g_i} \|_2^2 + Q_0} + \frac{\mu}{\tau_i \| h_{g_i} \|_2} \right) \tau_i  h_{g_i} = h_{g_i}.
\]
giving
\[
-e \tau_i \| h_{g_i} \|_2 + \mu \left(\frac{1}{2} \sum_{i=1}^m \tau_i^2 \| h_{g_i} \|_2^2 + Q_0 \right) = \| h_{g_i} \|_2 ~~~ \Forall i = 1, \cdots, m.
\]
Applying a summation from both sides, together with $\sum_{i=1}^m \tau_i \| h_{g_i} \|_2 = \xi$, yields
\lbeq{e.taui}
-e \xi + m \mu \left(\frac{1}{2} \sum_{i=1}^m \tau_i^2 \| h_{g_i} \|_2^2 + Q_0 \right) = \sum_{i=1}^m \| h_{g_i} \|_2,
\eeq
implying 
\[
\mu = \frac{2(\sum_{i=1}^m \|h_{g_i}\|_2 +e \xi)}{m(\sum_{i=1}^m \tau_i^2 \| h_{g_i} \|_2^2 + 2Q_0)}.
\]
By substituting this into (\ref{e.taui}), we have
\[
\tau_i = -\frac{1}{m e \| h_{g_i} \|_2}  \left( m \|h_{g_i}\|_2 - \sum_{i=1}^m \|h_{g_i}\|_2 - e \xi \right)
\]
leading to
\[
u = (\tau_i h_{g_1}, \cdots, \tau_m h_{g_m}).
\]
By substituting this into $E_{\gamma, h} (u) = e$, we get
\[
e = -\frac{\gamma+\langle h,u \rangle}{\frac{1}{2} \xi^2 + Q_0} = -\frac{2(\gamma + \sum_{i=1}^n \tau_i^2 \|h_{g_i}\|_2^2)}{\sum_{i=1}^n \tau_i^2 \|h_{g_i}\|_2^2 + 2Q_0},
\]
giving the result.
\end{proof}

\vspace{5mm}
\section{Numerical experiments}

A software package for solving unconstrained and simply constrained 
convex optimization problems with OSGA is publicly available at
\begin{center}
\url{http://homepage.univie.ac.at/masoud.ahookhosh/}.
\end{center}
The package is written in MATLAB; it uses the parameters 
\begin{equation*}
\delta = 0.9;~~ \alpha_{max} = 0.7;~~ \kappa = \kappa' = 0.5;~~ 
\Psi_{\mathrm{target}} = - \infty. 
\end{equation*}
and the prox-function (\ref{e.prox}) with 
$Q_0 = \frac{1}{2} \|x_0\|_2 + \epsilon$, where $\epsilon$ is the 
machine precision. A user manual \cite{AhoUM} describes the design and 
use of the package. Some examples are included as illustrations.

This section discusses numerical results and comparisons of OSGA with 
some state-of-the-art first-order solvers on some ridge regression and 
image deblurring problems. All numerical results were created with 
version 1.1 of the above software. The algorithms used for comparison 
use the default parameter values reported in the corresponding papers 
or packages. All numerical experiments were executed on a 
Toshiba Satellite Pro L750-176 laptop with Intel Core i7-2670QM 
processor and 8 GB RAM.

\subsection{Ridge regression}
In this subsection we consider a $l_2$-constrained least squares of 
the form (\ref{e.las2}) (so-called ridge regression, see \cite{HoeK}) and report some numerical results.

The problem is generated by
\[
\mathtt{[A,z,x] = i\_laplace(n),~~~ y= z + 0.1*rand,}
\]
where $n=5000$ is the problem dimension and $\mathtt{i\_laplace.m}$ is an ill-posed test problem generator using the inverse Laplace transformation from Regularization Tools MATLAB package, which is available in
\begin{center}
\url{http://www.imm.dtu.dk/~pcha/Regutools/}.
\end{center}
Since (\ref{e.las2}) is smooth and the projection on $C = \{x \in \mathbb{R}^n \mid \|x\| \leq \xi\}$ is available (see Table 1), we employ gradient projection algorithm (PGA), the spectral gradient projection \cite{BirMR} with the {\sc Grippo} et al. nonmonotone term \cite{GriLL} (SPG-G), the spectral gradient projection with the {\sc Amini} et al. nonmonotone term \cite{AmiAN} (SPG-A), and OSGA (see Proposition \ref{p.ssl2}) to solve this minimization problem. The parameters of SPG-G and SPG-A are the same as those reported in the associated papers, but SPG-A uses
\[
  \eta_k=\left\{
  \begin{array}{ll}
    \eta_0/2\ \ & \hbox{if\ \ $k=1$}, \\
    (\eta_{k-1}+\eta_{k-2})/2\ \ & \hbox{if\ \ $k \geq 2$}. \\
  \end{array}  \right. 
\]
The algorithms are stopped after 500 iterations.

\begin{table}[h] \label{t.l22l22}
\caption{Result summary for the ridge regression}
\begin{center}\footnotesize
\renewcommand{\arraystretch}{1.3}
\begin{tabular}{|l|l|l|l|l|l|}\hline
\multicolumn{1}{|l|}{} & \multicolumn{1}{l|}{{\bf $\xi$}}
&\multicolumn{1}{l|}{{\bf PGA}} & \multicolumn{1}{l|}{{\bf SPG-G}}
&\multicolumn{1}{l|}{{\bf SPG-A}} & \multicolumn{1}{l|}{{\bf OSGA}} \\ 
\hline
{\bf $f_b$} & $10$ & 101.70e-3 & 7.60e-3 & 6.41e-3 & 3.60e-3\\
{\bf Time(s) } &                    & 77.78 & 30.08 & 31.20 & 22.09\\
\hline
{\bf $f_b$} & $15$ & 48.23e-3 & 1.70e-3 & 1.31e-3 & 1.52e-3\\
{\bf Time(s) } &                    & 66.54 & 25.00 & 24.24 & 21.55\\
\hline
{\bf $f_b$} & $20$ & 23.08e-2 & 2.01e-2 & 1.74e-2 & 8.60e-3\\
{\bf Time(s) } &                    & 64.60 & 28.47 & 27.11 & 21.40\\
\hline
{\bf $f_b$} & $25$ & 23.00e-2 & 2.22e-2 & 1.24e-2 & 8.96e-3\\
{\bf Time(s) } &                    & 62.55 & 30.20 & 31.18 & 26.50\\
\hline
\end{tabular}
\end{center}
\end{table}

\begin{figure} \label{f.l22l22}
\centering
\subfloat[][$\delta_k$ versus iterations, $\xi = 10$]{\includegraphics[width=6.1cm]{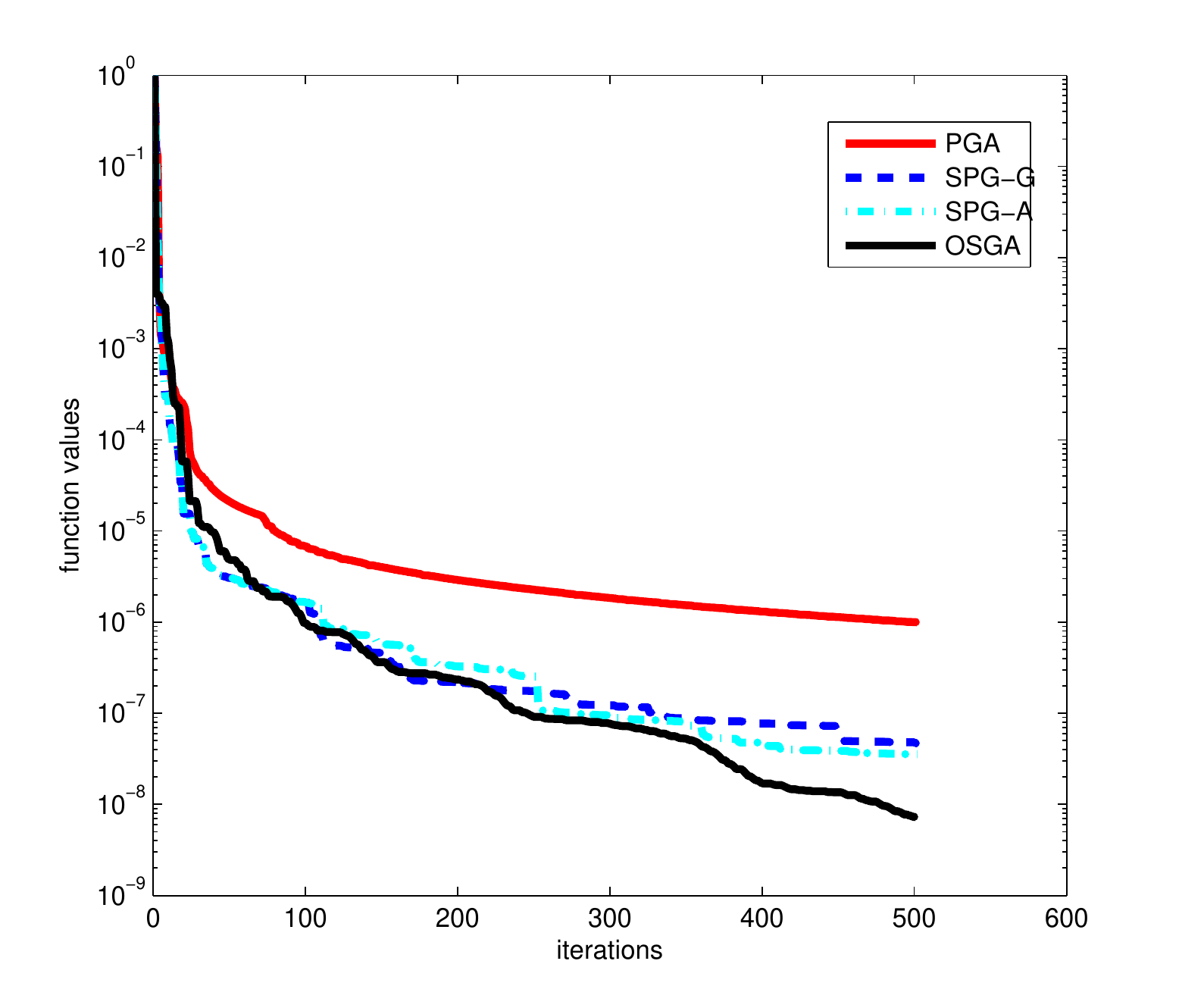}}%
\qquad
\subfloat[][$\delta_k$ versus iterations, $\xi = 15$]{\includegraphics[width=6.1cm]{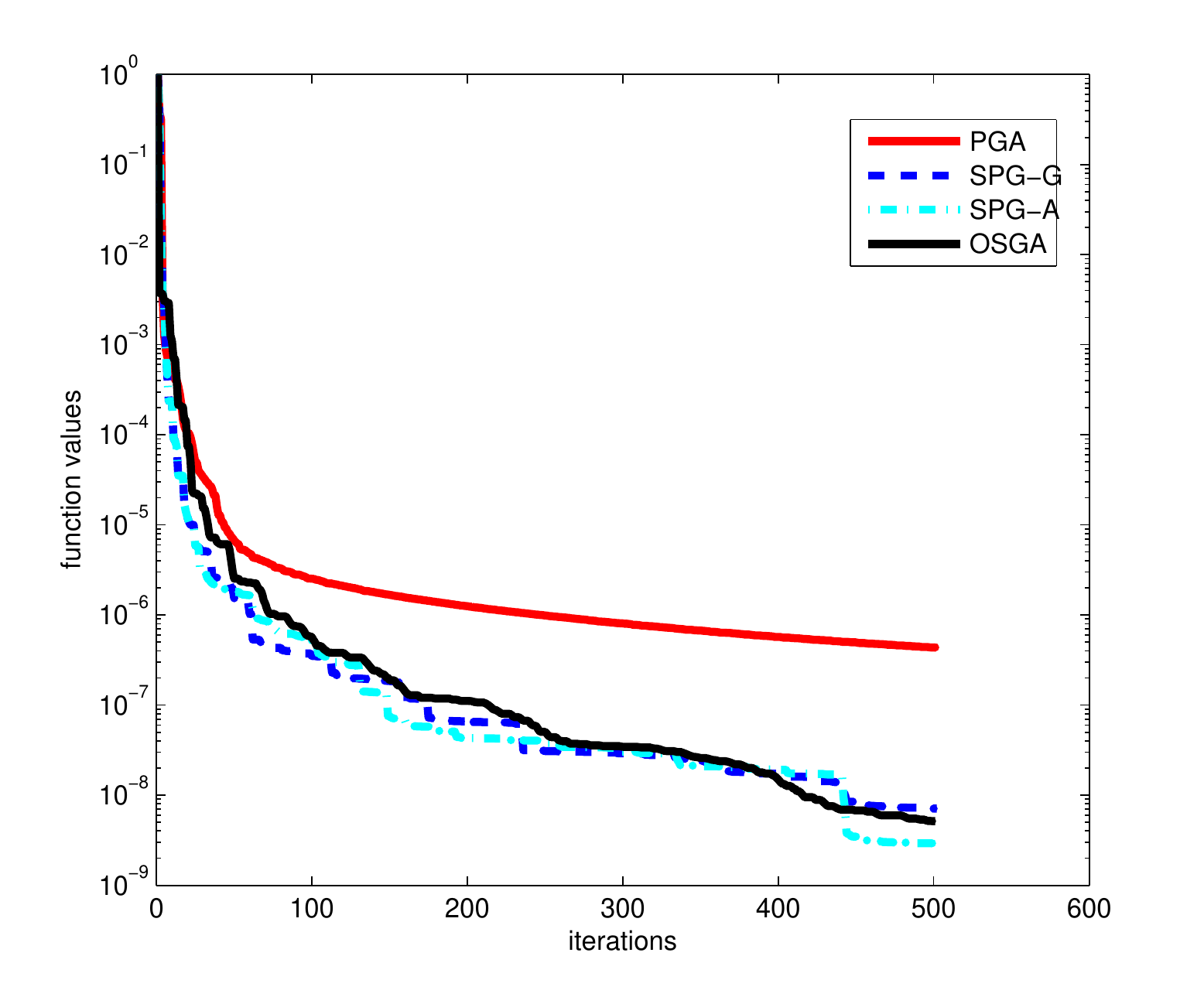}}
\qquad
\subfloat[][$\delta_k$ versus iterations, $\xi = 20$]{\includegraphics[width=6.1cm]{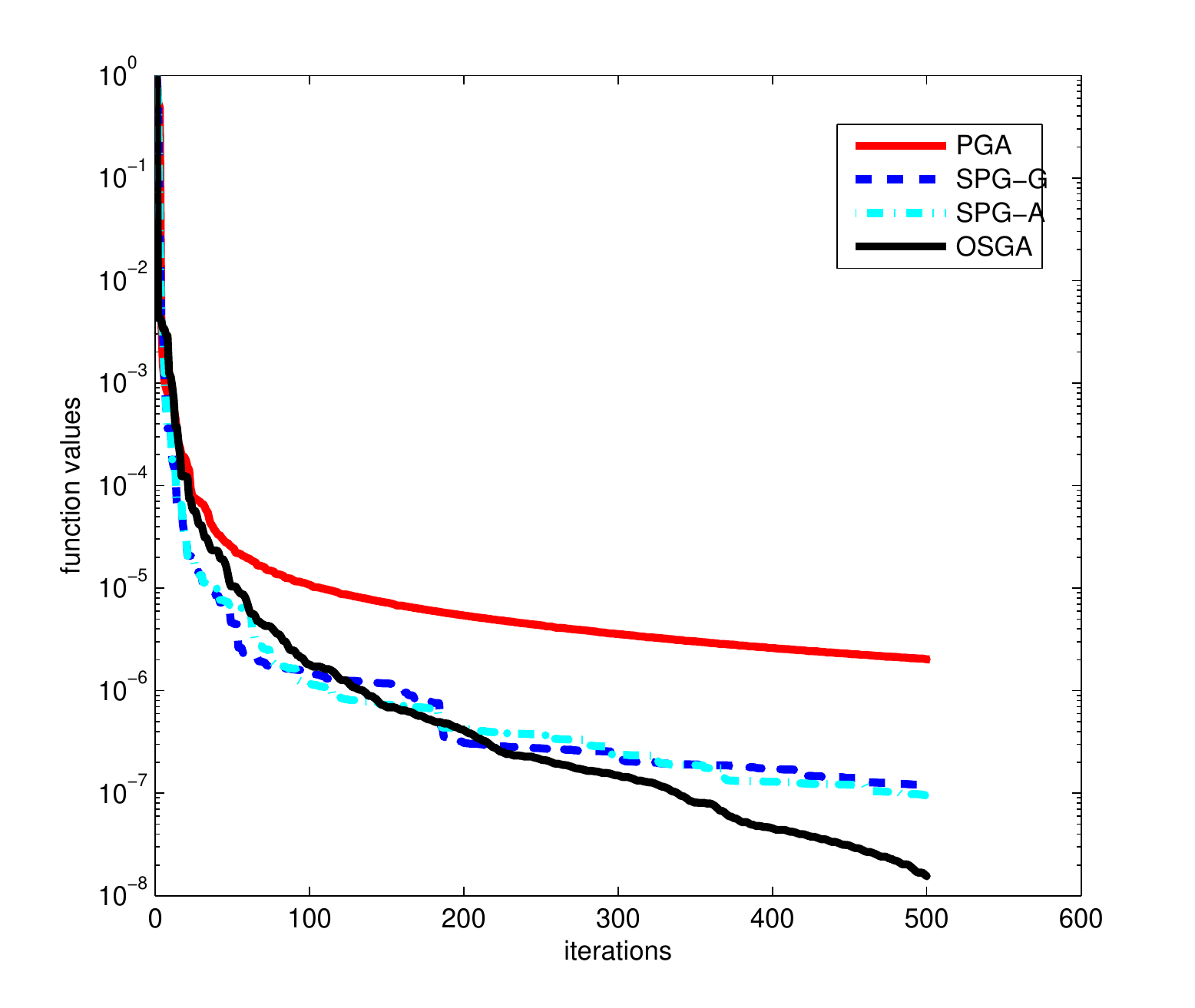}}%
\qquad
\subfloat[][$\delta_k$ versus iterations, $\xi = 25$]{\includegraphics[width=6.1cm]{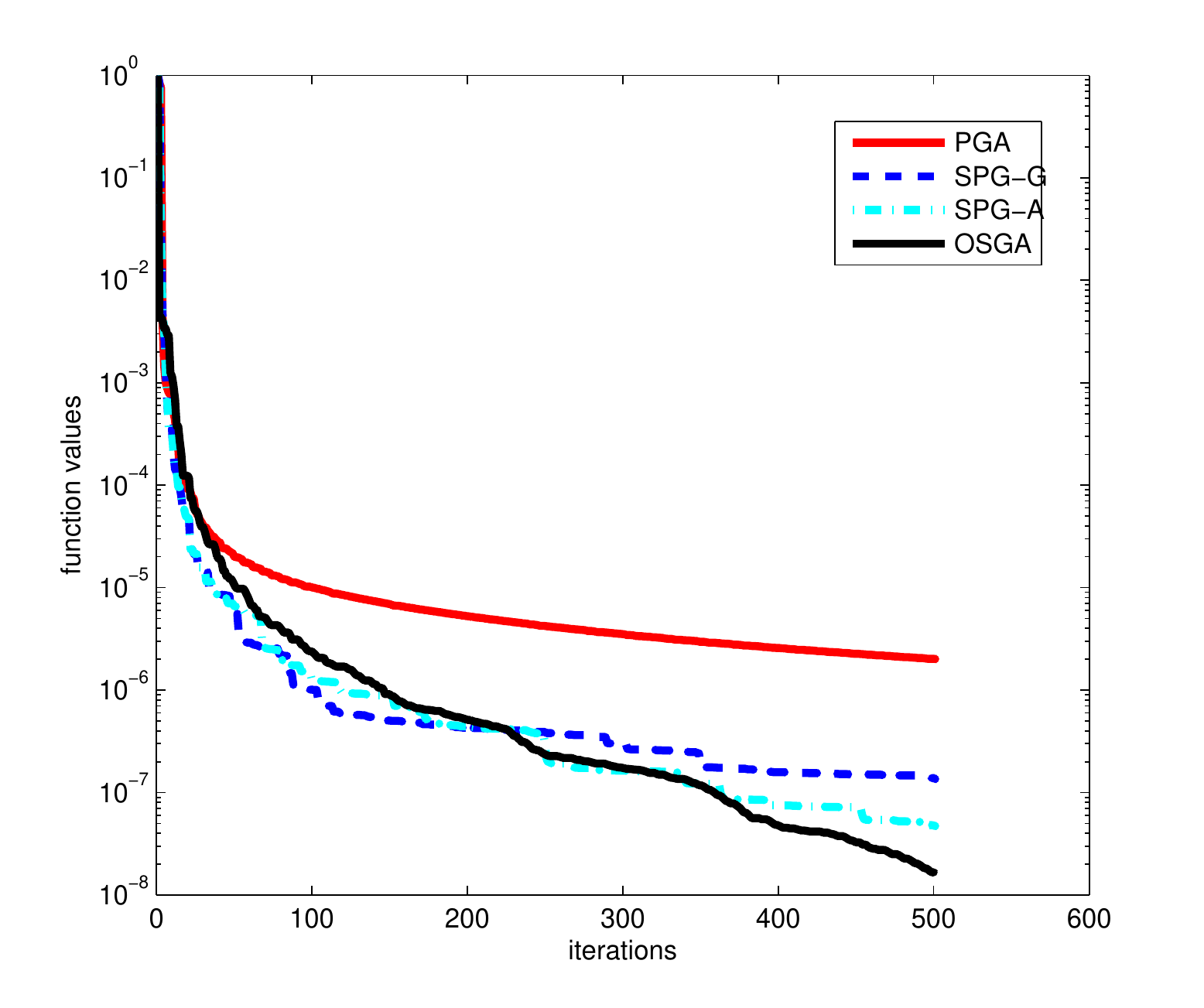}}

\caption{A comparison among PGA, SPG-G, SPG-A, and OSGA for solving the problem (\ref{e.las2}) based on the relative error of function values $\delta_k$ (\ref{e.delta}). The algorithms were stopped after 500 iterations.}
\end{figure}

In Table 3 we consider $\xi=10, 15, 20, 25$ and report the best attained function values and the running time. The results imply that OSGA attains the best running time and except for $\xi=15$ gives the best function values. To see the results of implementation in more details, we demonstrate the relative error of function values 
\lbeq{e.delta}
\delta_k := \frac{f_k - \widehat{f}}{f_0 - \widehat{f}}
\eeq
in Figure 1, where $\wh f$ denotes the minimum and $f_0$ shows the function value on an initial point $x_0$.

\subsection{Image deblurring with nonnegativity constraint}
As discussed in Section 3, inverse problems are appearing in many fields of applied sciences and Engineering. This is particularly happen when researchers use digital images to record and analyze results from experiments in many fields such as astronomy, medical sciences, biology, geophysics, and physics. In these cases, observing blurred and noisy images is a common phenomenon happening frequently because of environmental effects and imperfections in the imaging system. 

In many applications, the variable $x$ describes physical quantities, 
which is meaningful if each component of $x$ is restricted to be  
nonnegative. This constraint is referred as the nonnegativity 
constraint; it is especially useful for restoring blurred and noisy 
images, see \cite{BarV,KauN1,KauN2,Vog}.

We restore the $256 \times 256$ blurred and noisy MR-brain image using 
the model (\ref{e.l22itvr}) equipped with the isotropic total variation 
regularizer. The true image is available in 
\begin{center}
\url{http://graphics.stanford.edu/data/voldata/}.
\end{center}
The blurred/noisy image $y$ is generated by a $9 \times 9$ uniform blur and adding a Gaussian noise with zero mean and standard deviation 
set to $10^{-3}$. For restoring the image, we use OSGA (see Proposition \ref{p.nego}), MFISTA (a monotone version of 
FISTA proposed by {\sc Beck \& Teboulle} in \cite{BecT3}), ADMM (an alternating 
direction method proposed by {\sc Chan} et al. in \cite{ChaTY}), and PSGA (a projected subgradient scheme with nonsummable diminishing step size), see \cite{BoyXM}. The original codes of MFISTA and ADMM provided by the authors are used. Since  the methods are sensitive to the regularization parameter $\lambda$, three different regularization parameters are used. The algorithms are stopped after 100 iterations. The comparison concerning the quality of the recovered image is made via the so-called 
peak signal-to-noise ratio (PSNR) defined by 
\lbeq{e.psnr}
\mathrm{PSNR} = 20 \log_{10} \left( \frac{\sqrt{mn}}{\|x - x_t\|_F} \right)
\eeq
and the improvement in signal-to-noise ratio (ISNR) defined by
\lbeq{e.isnr}
\mathrm{ISNR} = 20 \log_{10} \left( \frac{\|y - x_t\|_F}{\|x - x_t\|_F} \right), 
\eeq
where $\|\cdot\|_F$ is the Frobenius norm, $x_t$ denotes the $m \times n$ true image, $y$ is the observed image, and pixel values are in $[0, 1]$. The results of implementation are summarized in Table 4 and Figures 2 and 3.

\begin{table}[h] \label{t.l22itvr}
\caption{Result summary for L22ITV}
\begin{center}\footnotesize
\renewcommand{\arraystretch}{1.3}
\begin{tabular}{|l|l|l|l|l|l|}\hline
\multicolumn{1}{|l|}{} & \multicolumn{1}{l|}{{\bf $\lambda$}}
&\multicolumn{1}{l|}{{\bf PSGA}} & \multicolumn{1}{l|}{{\bf MFISTA}}
&\multicolumn{1}{l|}{{\bf ADMM}} & \multicolumn{1}{l|}{{\bf OSGA}} \\ 
\hline
{\bf PSNR}  &                    & 32.59 & 32.67 & 32.66 & 32.73\\
{\bf $f_b$} & $5 \times 10^{-4}$ & 0.3528 & 0.3079 & 0.3080 & 0.3149\\
{\bf Time(s) } &                    & 1.14 & 7.61 & 1.11 & 1.82\\
\hline
{\bf PSNR}  &                    & 33.23 & 33.96 & 33.95 & 33.97\\
{\bf $f_b$} & $1 \times 10^{-4}$ & 0.1184 & 0.0960 & 0.0958 & 0.0980\\
{\bf Time(s) } &                    & 1.14 & 7.34 & 1.04 & 1.71\\
\hline
{\bf PSNR}  &                    & 33.24 & 34.45 & 34.49 & 34.46\\
{\bf $f_b$} & $5 \times 10^{-5}$ & 0.1174 & 0.0653 & 0.0651 & 0.0669\\
{\bf Time(s) } &                    & 1.15 & 6.51 & 1.06 & 1.67\\
\hline
\end{tabular}
\end{center}
\end{table}

\begin{figure} \label{f.deb1}
\centering
\subfloat[][$\delta_k$ versus iterations, $\lambda = 5 \times 10^{-4}$]{\includegraphics[width=6.1cm]{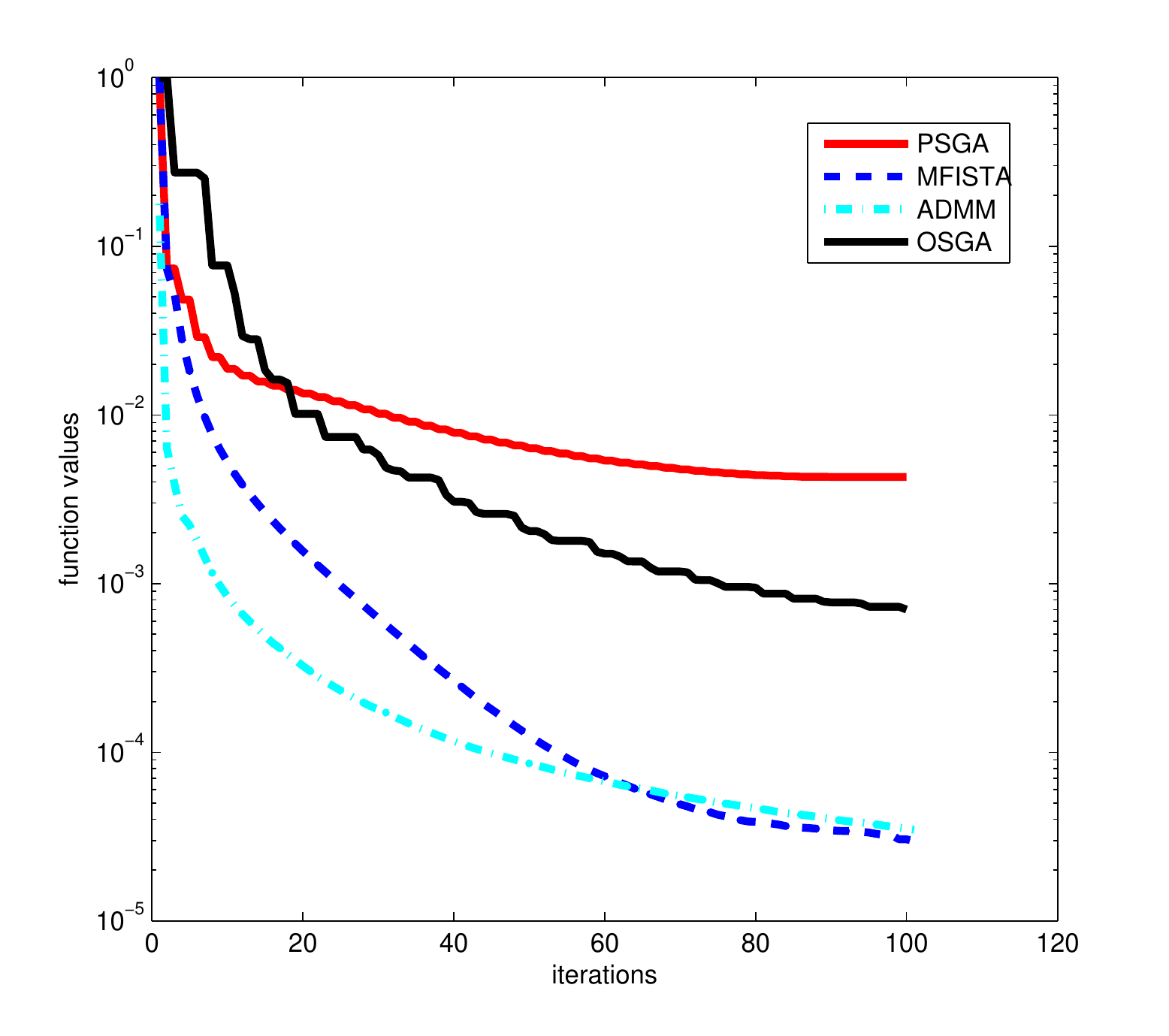}}%
\qquad
\subfloat[][ISNR versus iterations, $\lambda = 5 \times 10^{-4}$]{\includegraphics[width=6.1cm]{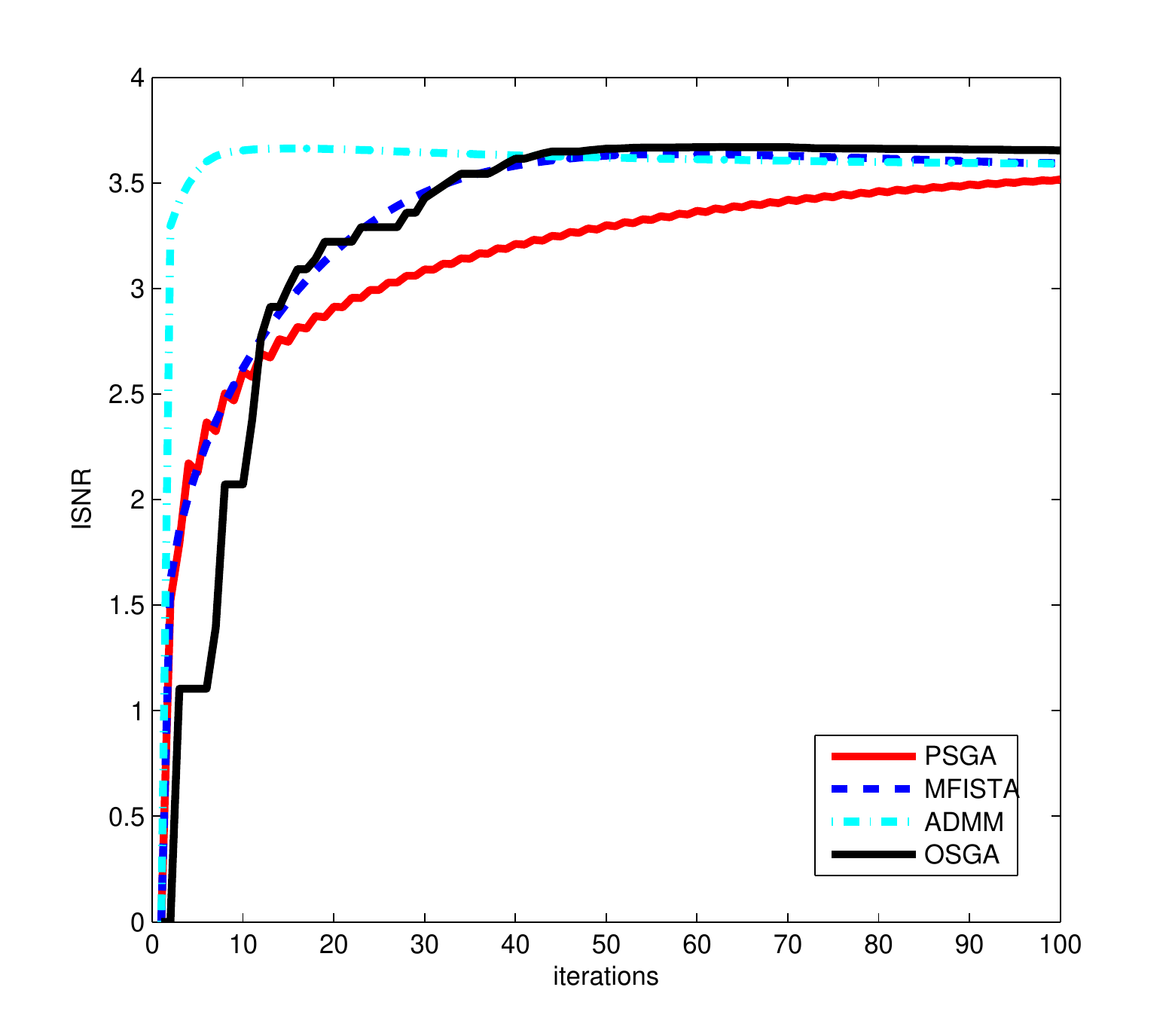}}
\qquad
\subfloat[][$\delta_k$ versus iterations, $\lambda = 1 \times 10^{-4}$]{\includegraphics[width=6.1cm]{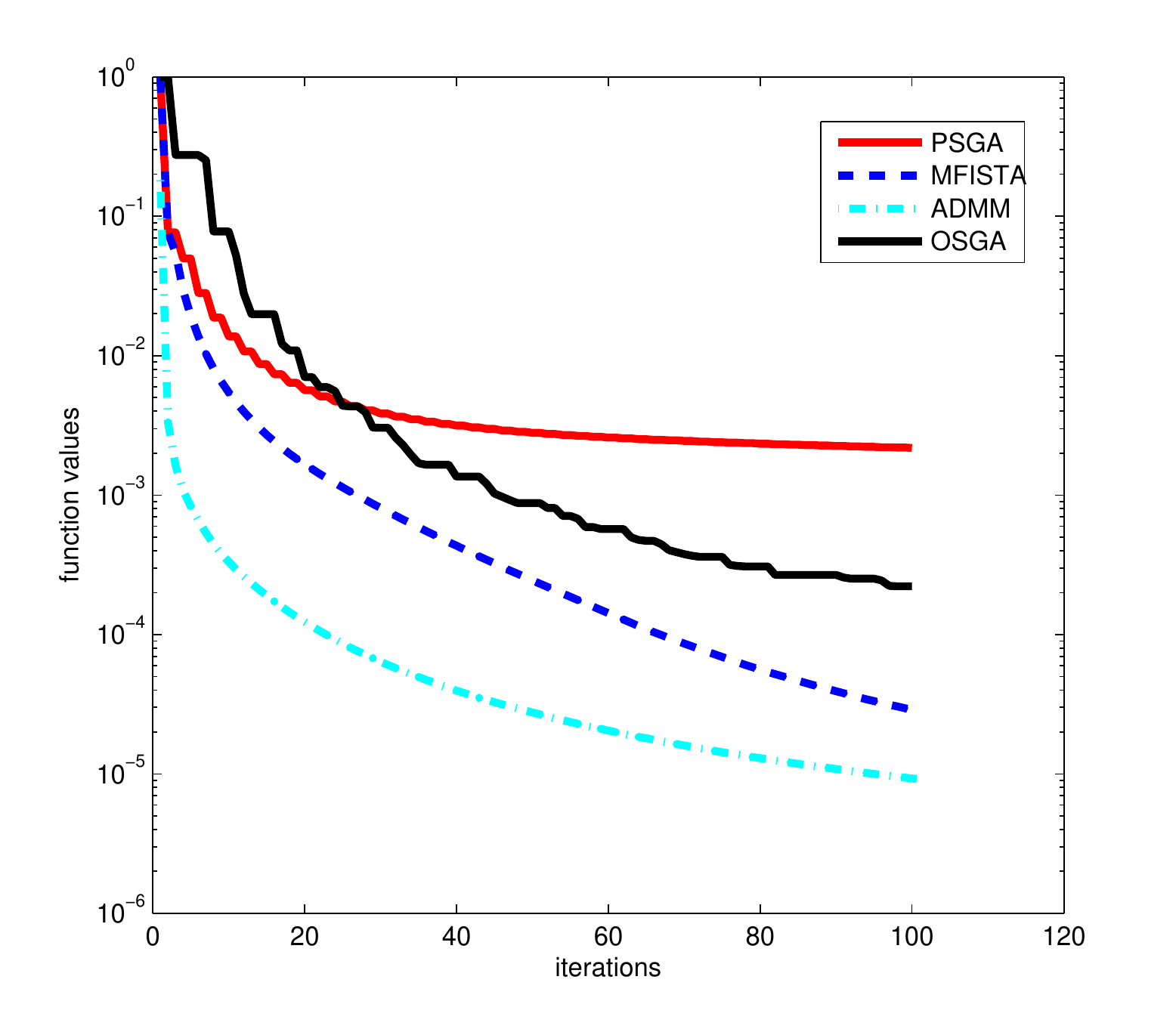}}%
\qquad
\subfloat[][ISNR versus iterations, $\lambda = 1 \times 10^{-4}$]{\includegraphics[width=6.1cm]{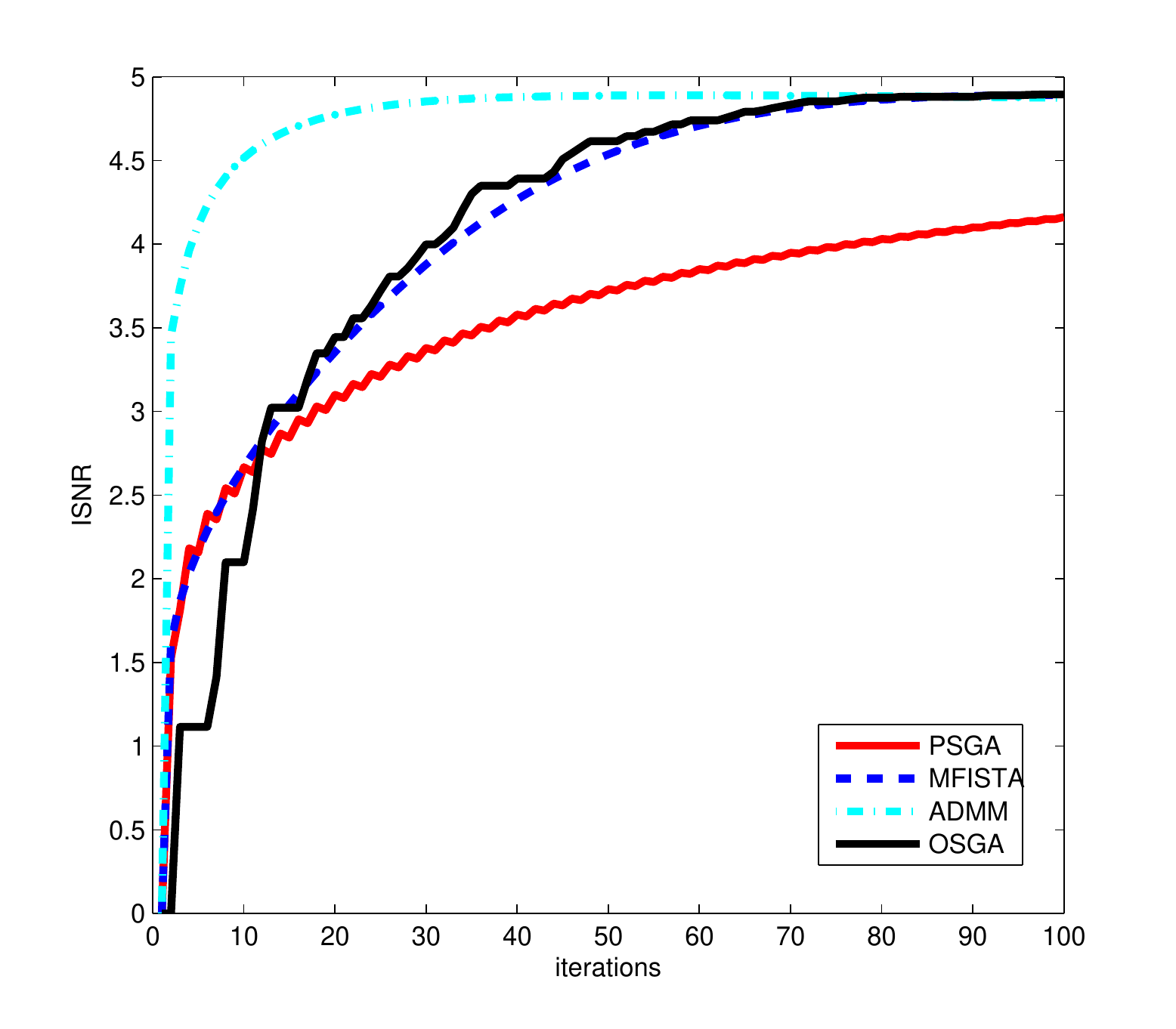}}
\qquad
\subfloat[][$\delta_k$ versus iterations, $\lambda = 5 \times 10^{-5}$]{\includegraphics[width=6.1cm]{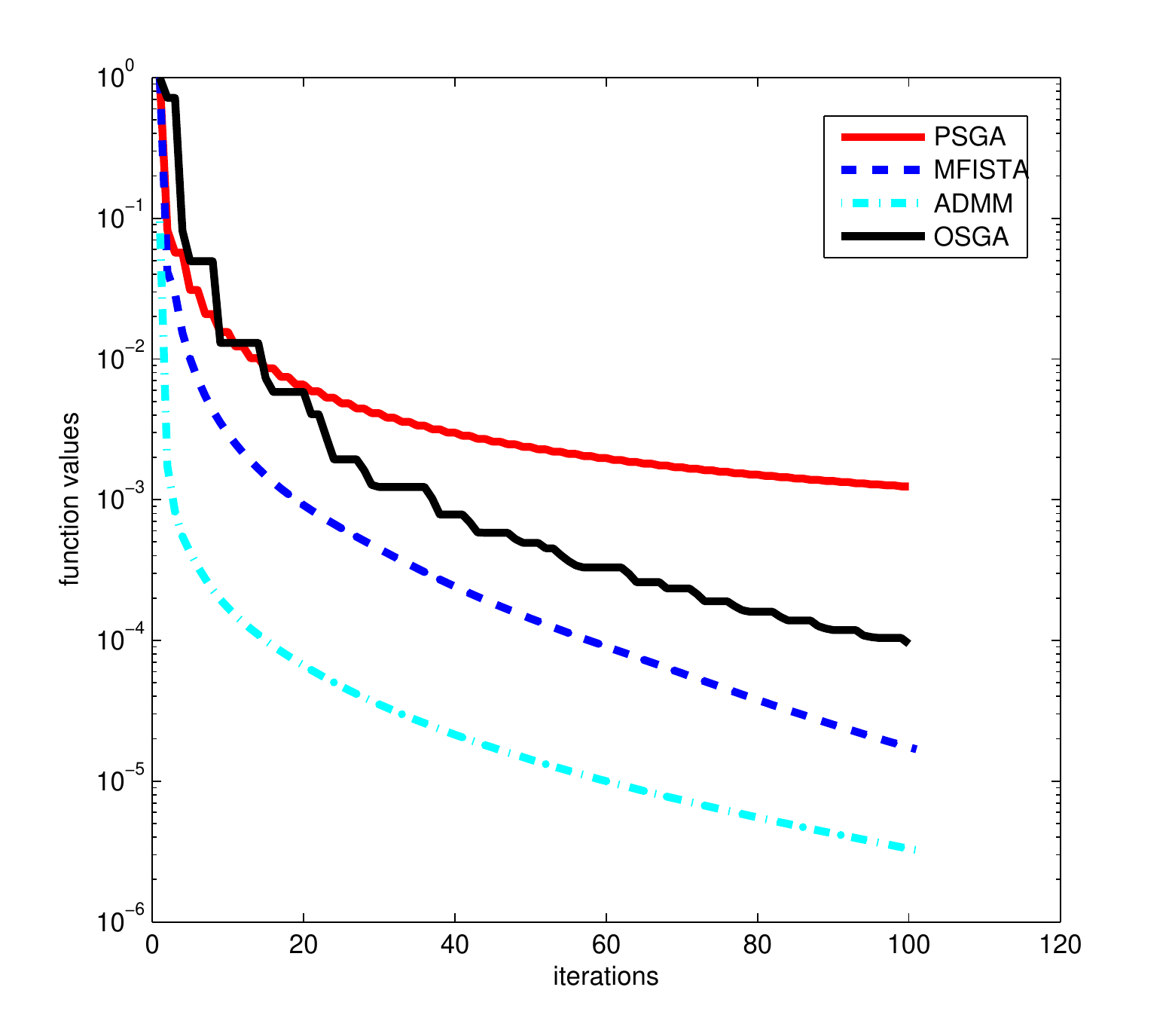}}%
\qquad
\subfloat[][ISNR versus iterations, $\lambda = 5 \times 10^{-5}$]{\includegraphics[width=6.1cm]{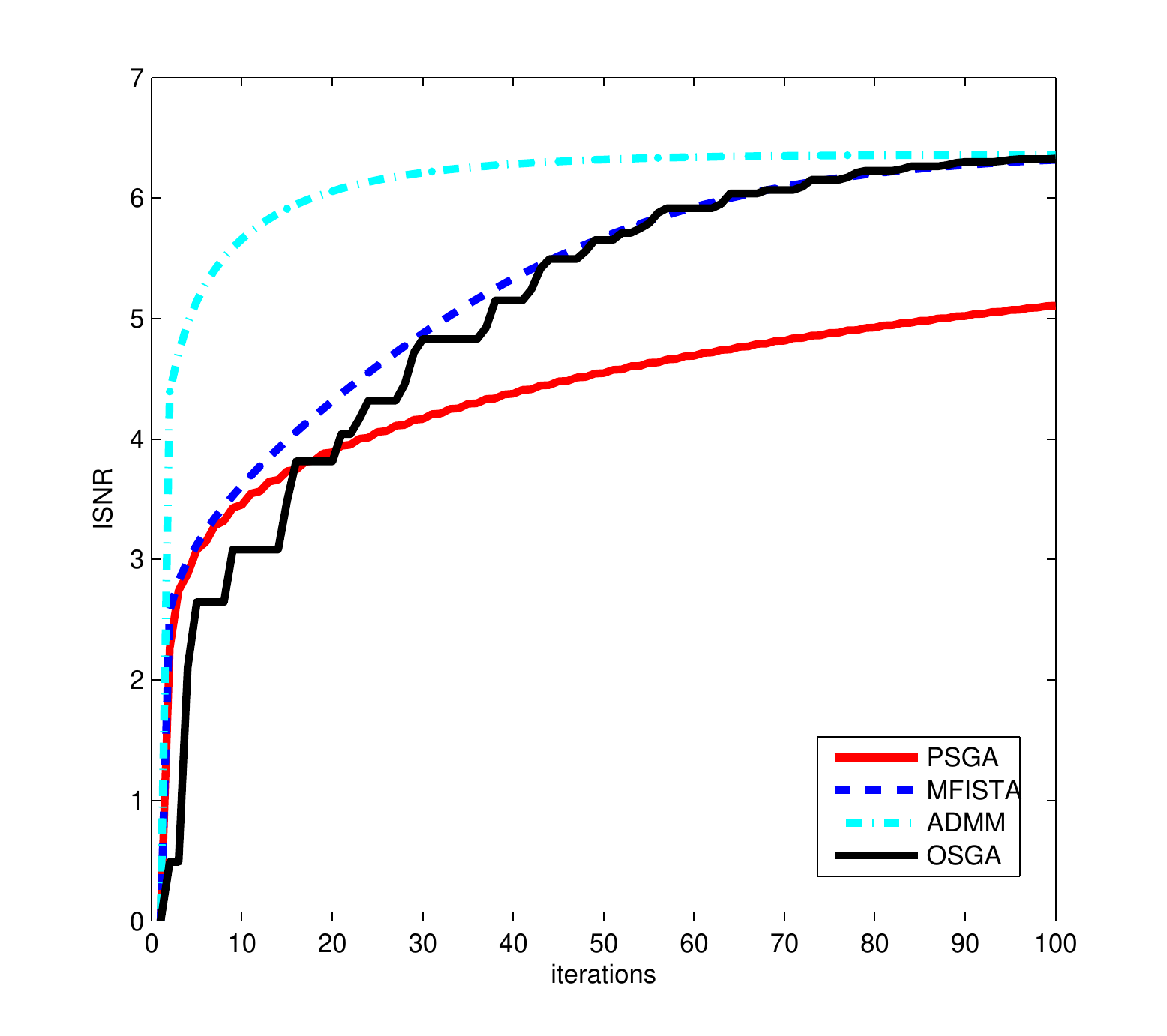}}

\caption{A comparison among PSGA, MFISTA, ADMM, and OSGA for deblurring the $256 \times 256$ MR-brain image with the $9 \times 9$ uniform blur and the Gaussian noise with deviation $10^{-3}$. The algorithms were stopped after 100 iterations. Subfigures (a), (c), and (e) display the relative error of function values $\delta_k$ (\ref{e.delta}) versus iterations, and Subfigures (b), (d), and (f) display ISNR (\ref{e.isnr}) versus iterations.}
\end{figure}

\begin{figure} \label{f.deb2}
\centering
\subfloat[][Original image]{\includegraphics[width=6.1cm]{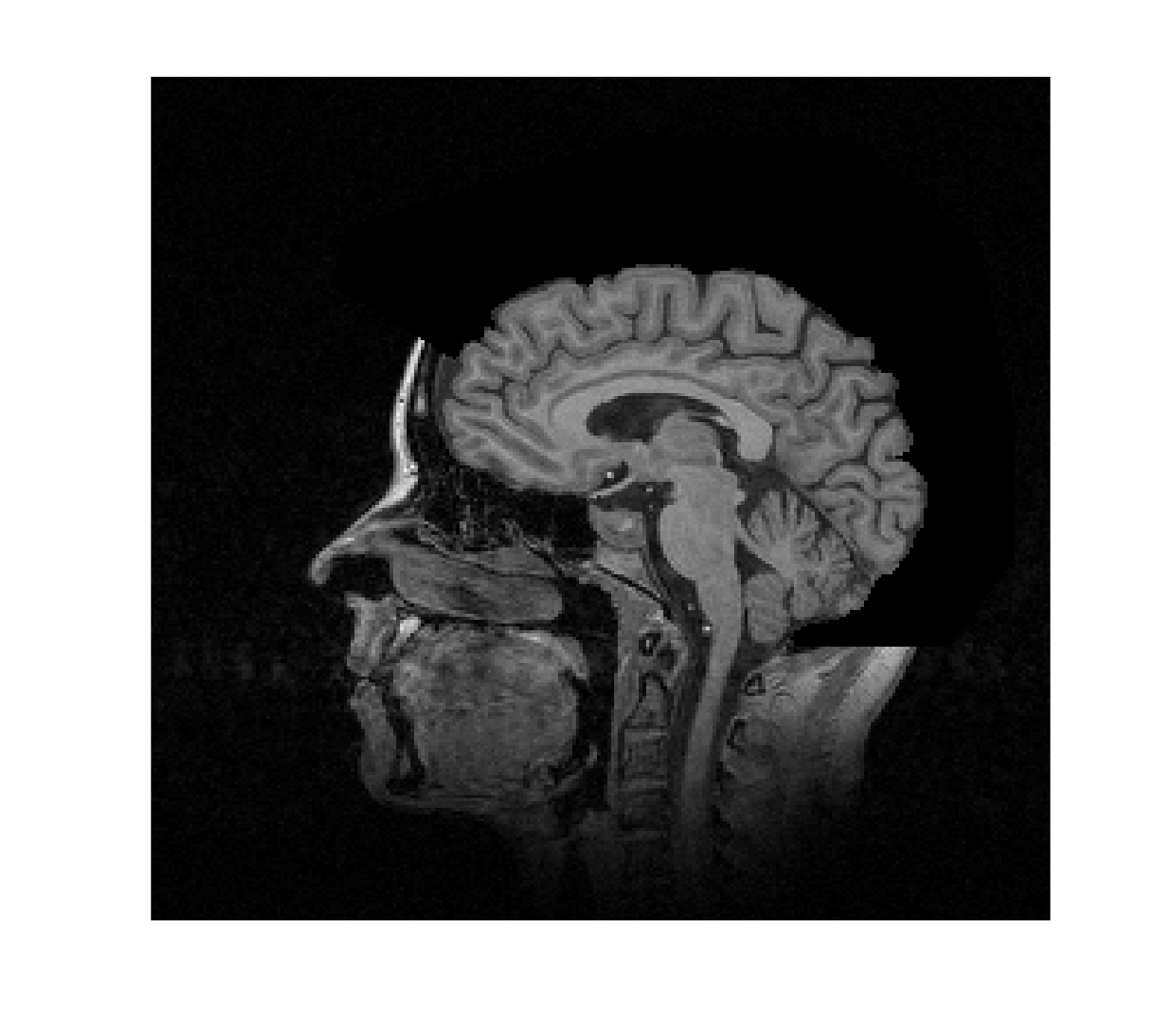}}%
\qquad
\subfloat[][Blurred/noisy image]{\includegraphics[width=6.1cm]{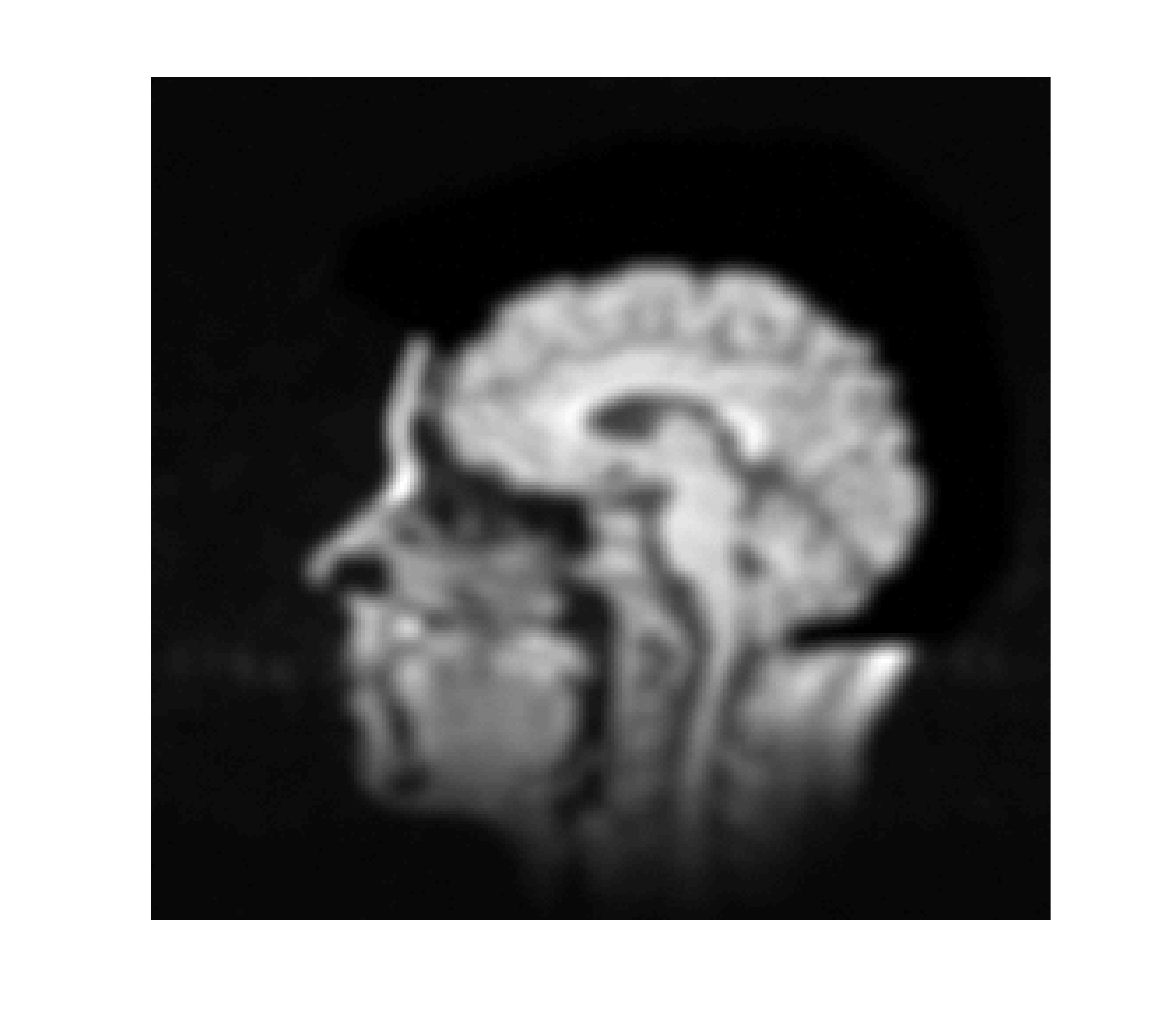}}
\qquad
\subfloat[][PSGA: $f = 0.1174, \mathrm{PSNR} = 33.24, \mathrm{T} = 1.15$]{\includegraphics[width=6.1cm]{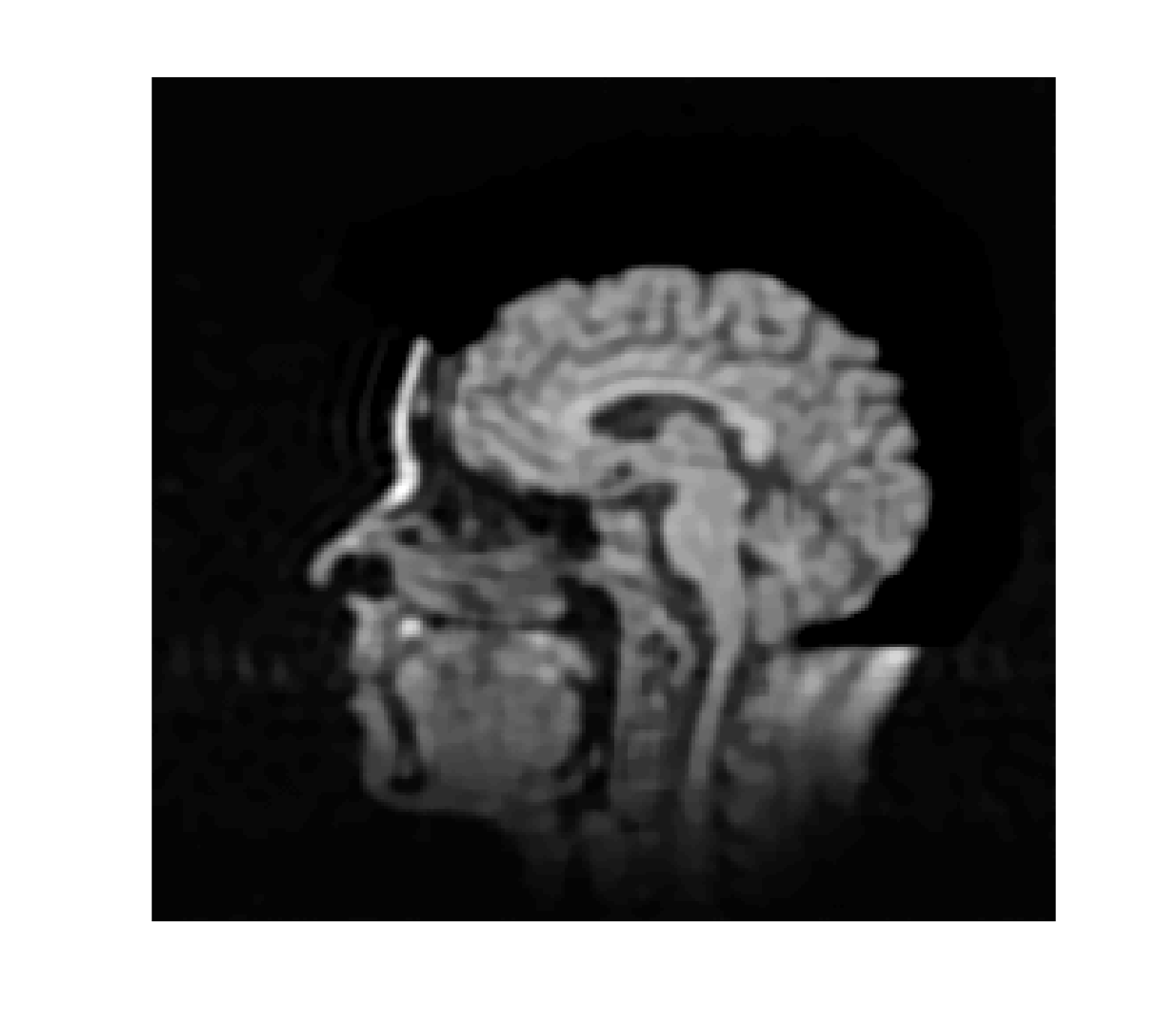}}%
\qquad
\subfloat[][MFISTA: $f = 0.0653, \mathrm{PSNR} = 34.45, \mathrm{T} = 6.51$]{\includegraphics[width=6.1cm]{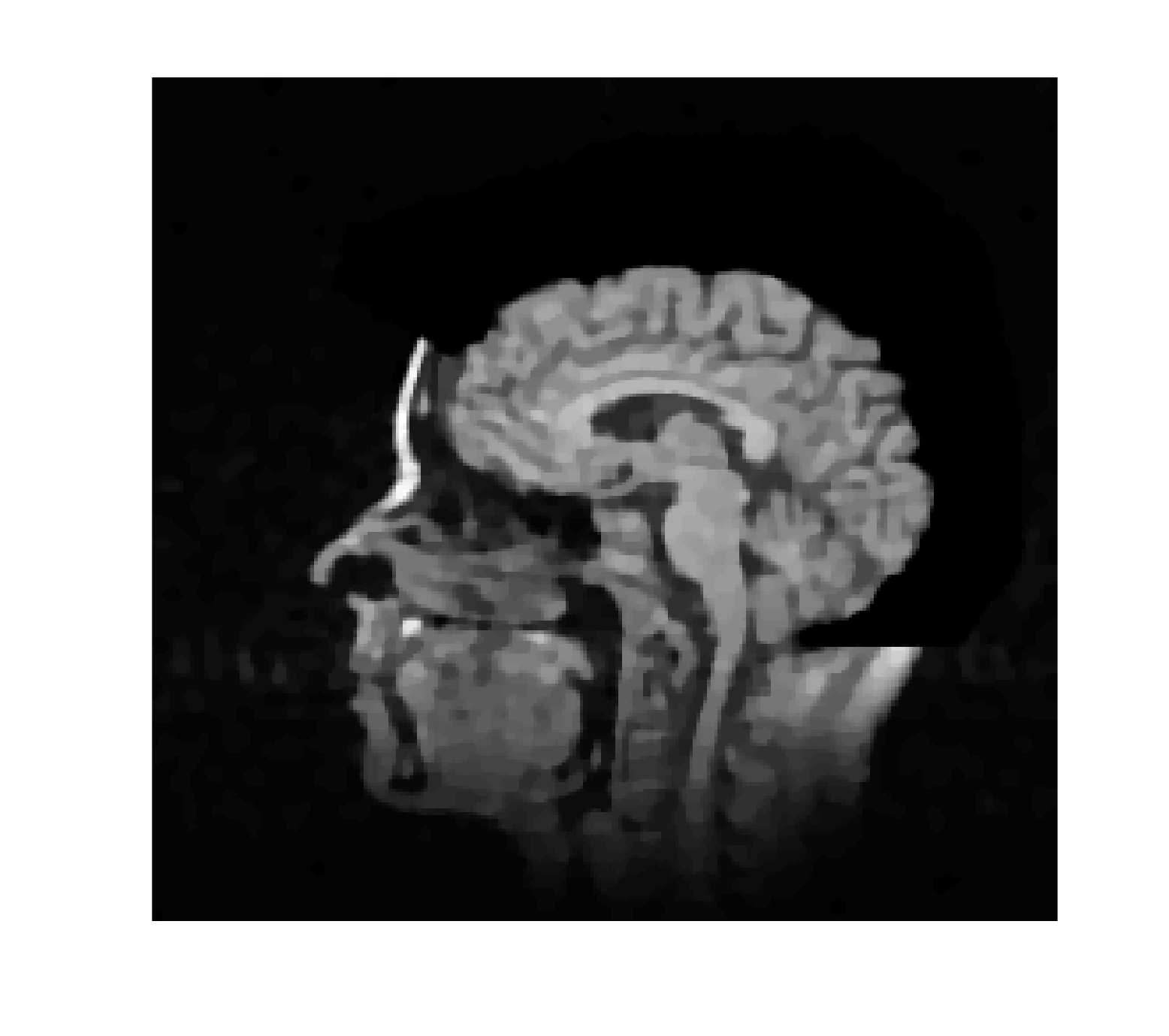}}
\qquad
\subfloat[][ADMM: $f = 0.0651, \mathrm{PSNR} = 34.49, \mathrm{T} = 1.06$]{\includegraphics[width=6.1cm]{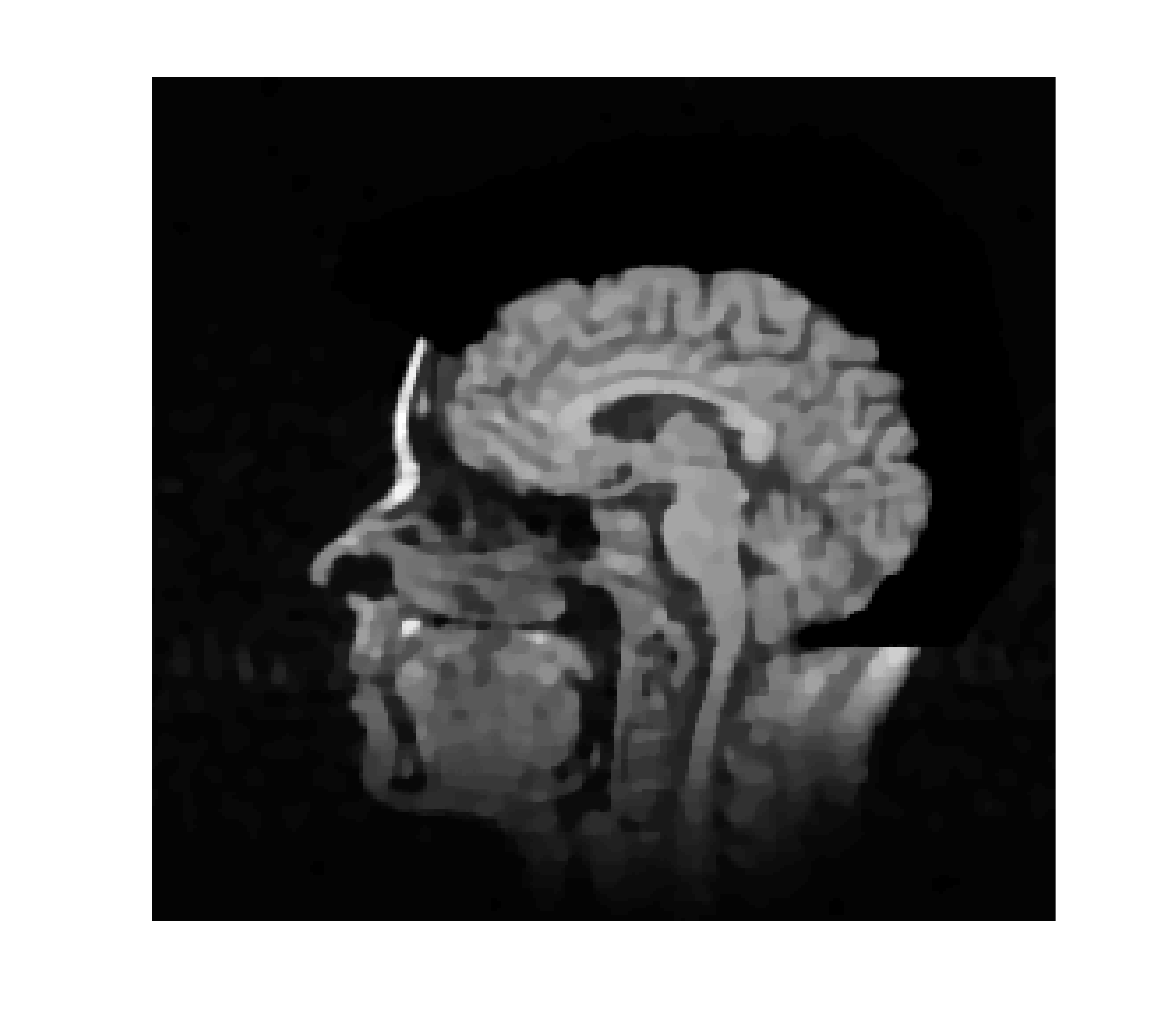}}%
\qquad
\subfloat[][OSGA: $f = 0.0669, \mathrm{PSNR} = 34.46, \mathrm{T} = 1.97$]{\includegraphics[width=6.1cm]{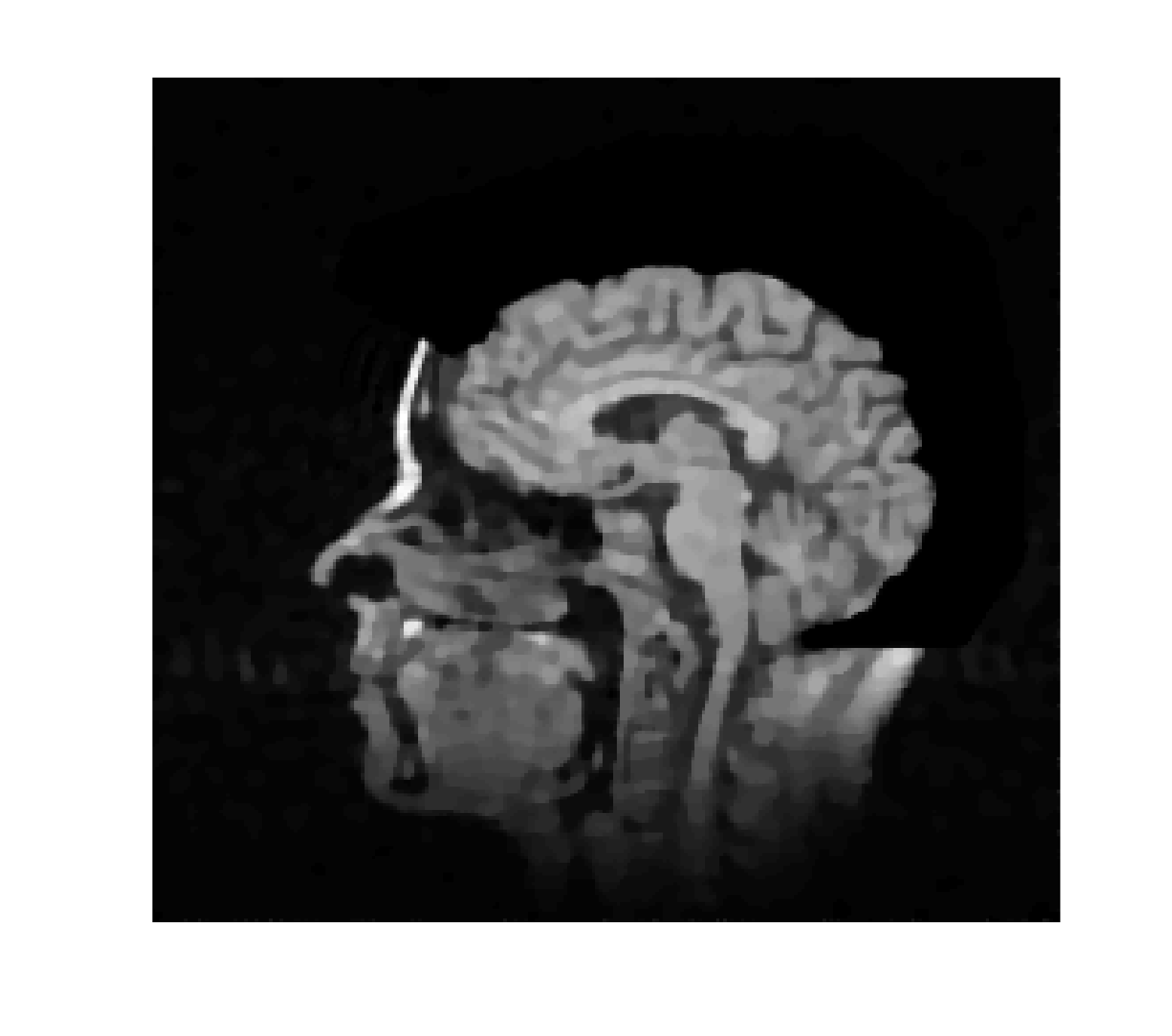}}
\caption{Deblurring of the $256 \times 256$ MR-brain image with the $9 \times 9$ uniform blur and the Gaussian noise with deviation $10^{-3}$ by PSGA, MFISTA, ADMM, and OSGA with the regularization parameter $\lambda = 10^{-4}$. The algorithms were stopped after 100 iterations.}%
\end{figure}

In Table 4 we report PSNR, the best available approximation $f_b$ of the minimimum, and the running time in seconds for three different regularization parameters. The results reported in Figure 2 regarding function values and ISNR show that the algorithms considered are sensitive to the parameter $\lambda$, however, the best results obtained for $\lambda=10^{-4}$. More specifically, the results about function values in subfigures (a), (c), and (e) demonstrate that OSGA outperforms PSGA, which means it performs much better than the lower complexity bound $O(\varepsilon^{-2})$, however, it cannot perform similar to MFISTA attaining the complexity of the order $O(\varepsilon^{-1/2})$. Subfigures (b), (d), and (f) 
show that OSGA is comparable with MFISTA and ADMM and even better than them in the sense of ISNR. The deblurred images by the algorithms considered are illustrated in Figure 3 for $\lambda =10^{-4}$. 

We also consider the restoration of the $641 \times 641$  blurred/noisy Dione image using (\ref{e.l1itvr}). The true image is available in
\begin{center}
\url{http://photojournal.jpl.nasa.gov/Help/
ImageGallery.html}.
\end{center} 
The blurred/noisy image is constructed from the $7 \times 7$ Gaussian kernel with standard deviation 5 and salt-and-pepper impulsive noise with the level $50 \%$. To recover the image, we use DRPD-1, DRPD-2 (Douglas-Rachford primal-dual schemes proposed by {\sc Bo? \& Hendrich} in \cite{BotH2}), ADMM, and OSGA. The algorithms are stopped after 100 iterations, and three different regularization parameters are considered. The results of implementation are reported in Table 5 and Figures 4 and 5. 

The results of Table 5 shows that OSGA outperforms the others in the sense of PSNR. Figure 4 indicates that OSGA attains the best function values for $\lambda = 10^{-1}$ and $\lambda = 5 \times 10^{-2}$, however, ADMM get the best function value for $\lambda = 5 \times 10^{-1}$. It also implies that OSGA are comparable or even better that the others regarding ISNR. The resulted images for $\lambda = 10^{-1}$ are illustrated in Figure 5, demonstrating that the algorithms can restore the image by acceptable qualities while OSGA obtains the best function value and PSNR.

\begin{table}[!htbp]
\caption{Results summary for L1ITV}
\begin{center}\footnotesize
\renewcommand{\arraystretch}{1.3}
\begin{tabular}{|l|l|l|l|l|l|}\hline
\multicolumn{1}{|l|}{} & \multicolumn{1}{l|}{{\bf $\lambda$}}
&\multicolumn{1}{l|}{{\bf DRPD-1}} & \multicolumn{1}{l|}{{\bf DRPD-2}}
&\multicolumn{1}{l|}{{\bf ADMM}} & \multicolumn{1}{l|}{{\bf OSGA}} \\ 
\hline
{\bf PSNR}  &                    & 37.43 & 36.66 & 37.42 & 37.50\\
{\bf $f_b$} & $5 \times 10^{-1}$ & 1.0352e+5 & 1.0365e+5 & 1.0293e+5 & 1.0326e+5\\
{\bf Time } &                    & 10.86 & 6.83 & 8.57 & 9.01\\
\hline
{\bf PSNR}  &                    & 38.70 & 38.11 & 38.35 & 38.73\\
{\bf $f_b$} & $1 \times 10^{-1}$ & 1.0324e+5 & 1.0294e+5 & 1.0281e+5 & 1.0281e+5\\
{\bf Time } &                    & 10.43 & 6.68 & 8.46 & 8.32\\
\hline
{\bf PSNR}  &                    & 37.09 & 36.77 & 30.06 & 37.06\\
{\bf $f_b$} & $5 \times 10^{-2}$ & 1.0336e+5 & 1.0321e+5 & 1.0312e+5 & 1.0299e+5\\
{\bf Time } &                    & 10.26 & 6.27 & 8.25 & 9.23\\
\hline
\end{tabular}
\end{center}
\end{table}

\begin{figure} \label{f.deb3}
\centering
\subfloat[][$\delta_k$ versus iterations, $\lambda = 5 \times 10^{-1}$]{\includegraphics[width=6.1cm]{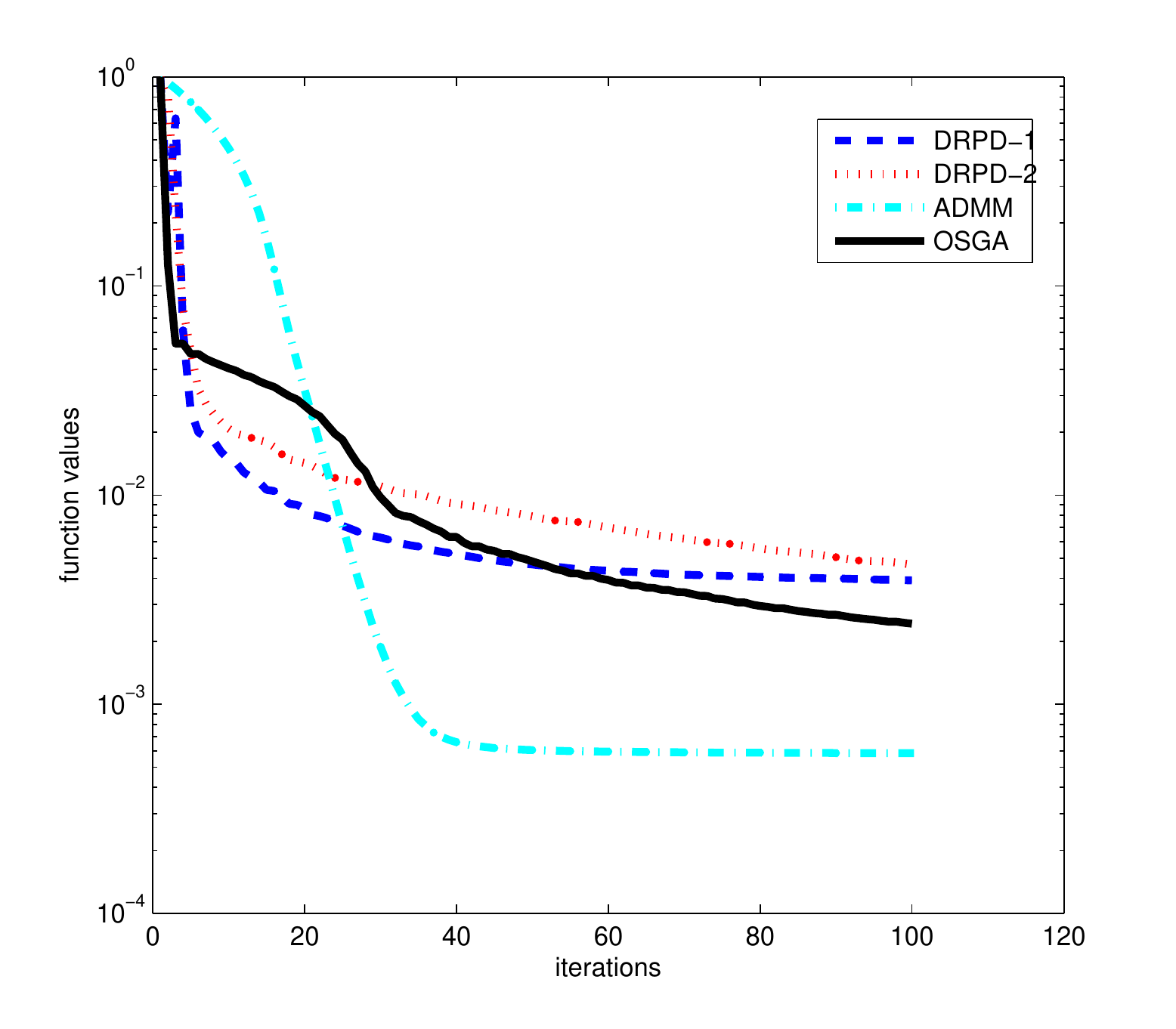}}%
\qquad
\subfloat[][ISNR versus iterations, $\lambda = 5 \times 10^{-1}$]{\includegraphics[width=6.1cm]{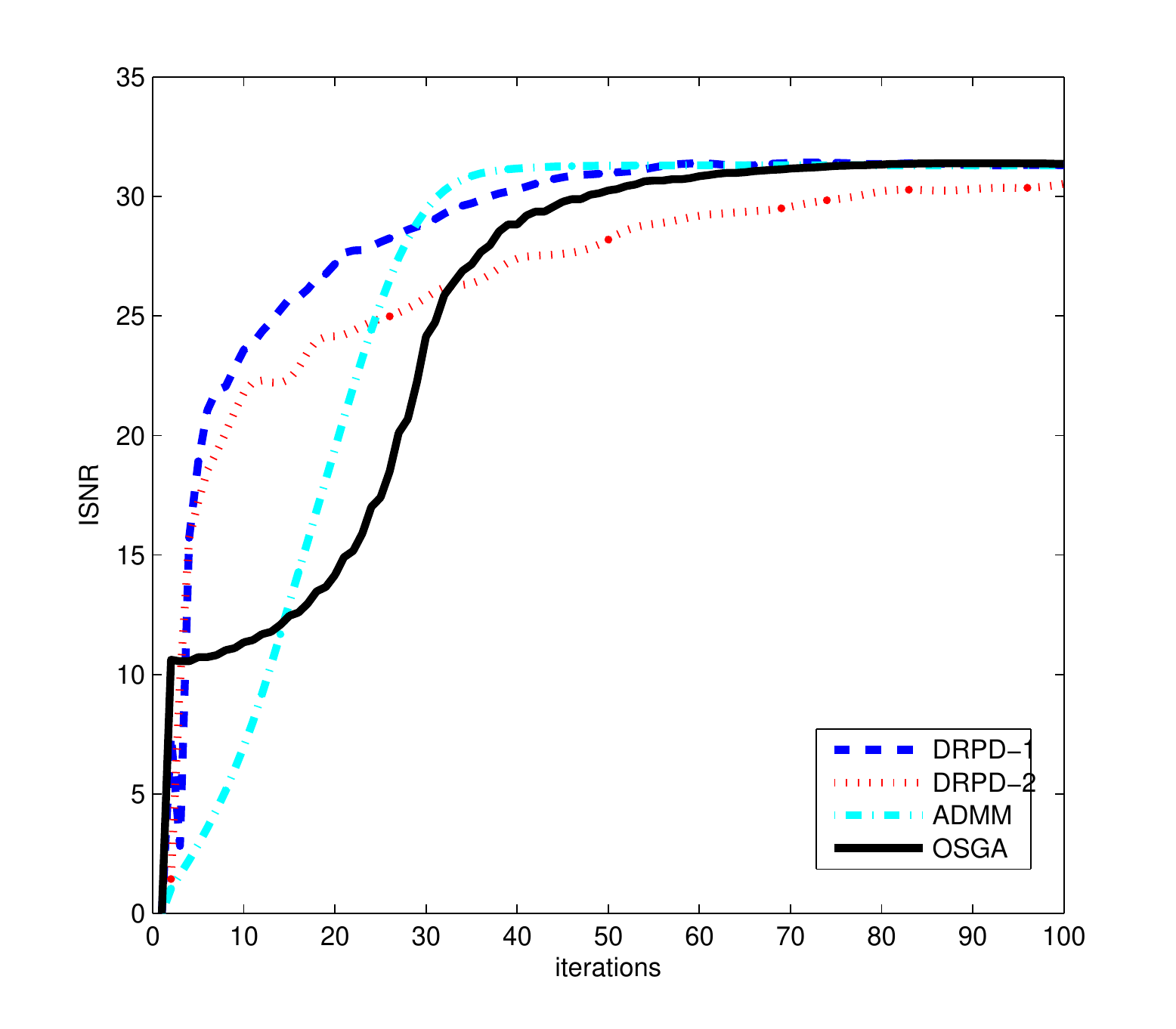}}
\qquad
\subfloat[][$\delta_k$ versus iterations, $\lambda = 1 \times 10^{-1}$]{\includegraphics[width=6.1cm]{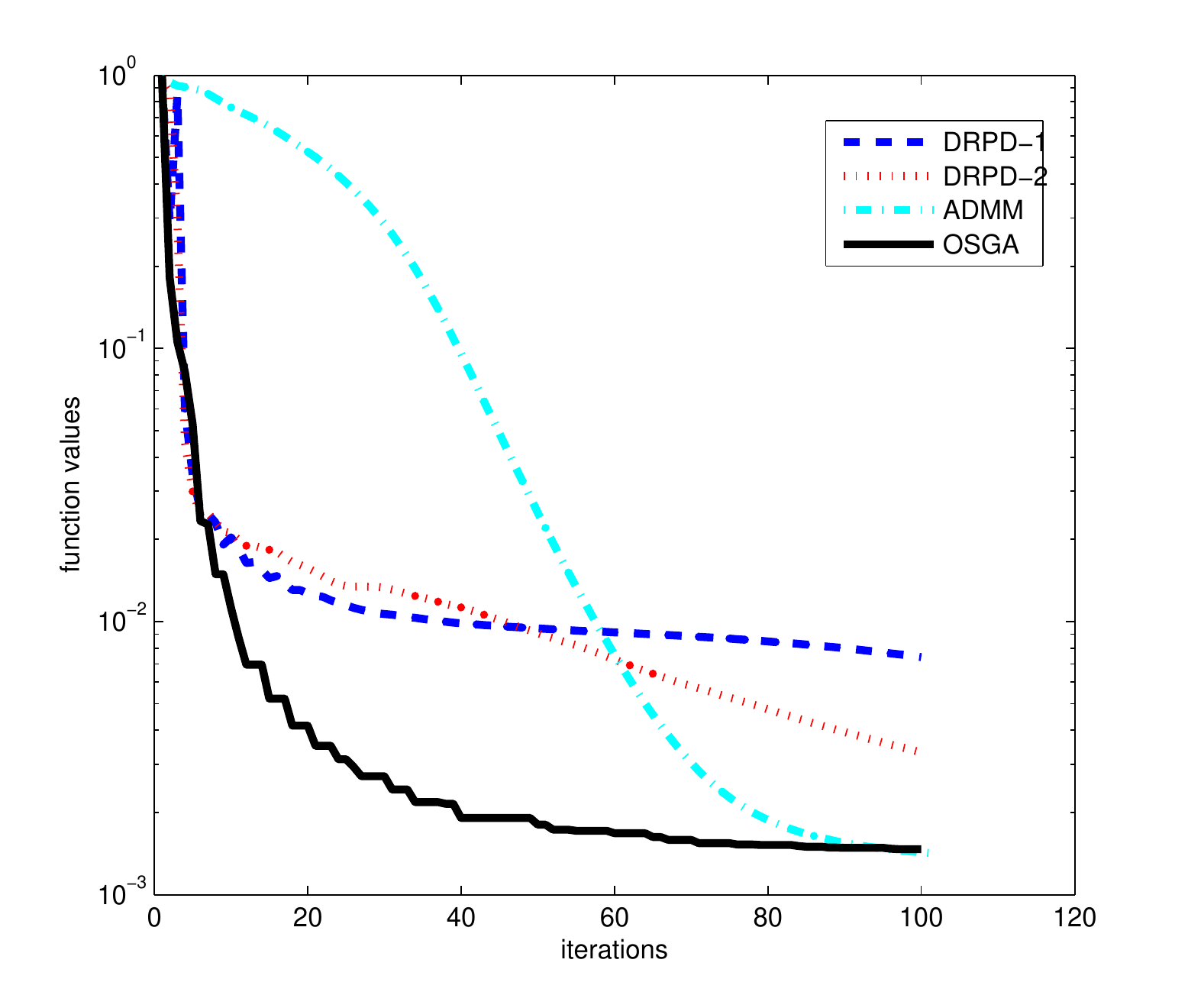}}%
\qquad
\subfloat[][ISNR versus iterations, $\lambda = 1 \times 10^{-1}$]{\includegraphics[width=6.1cm]{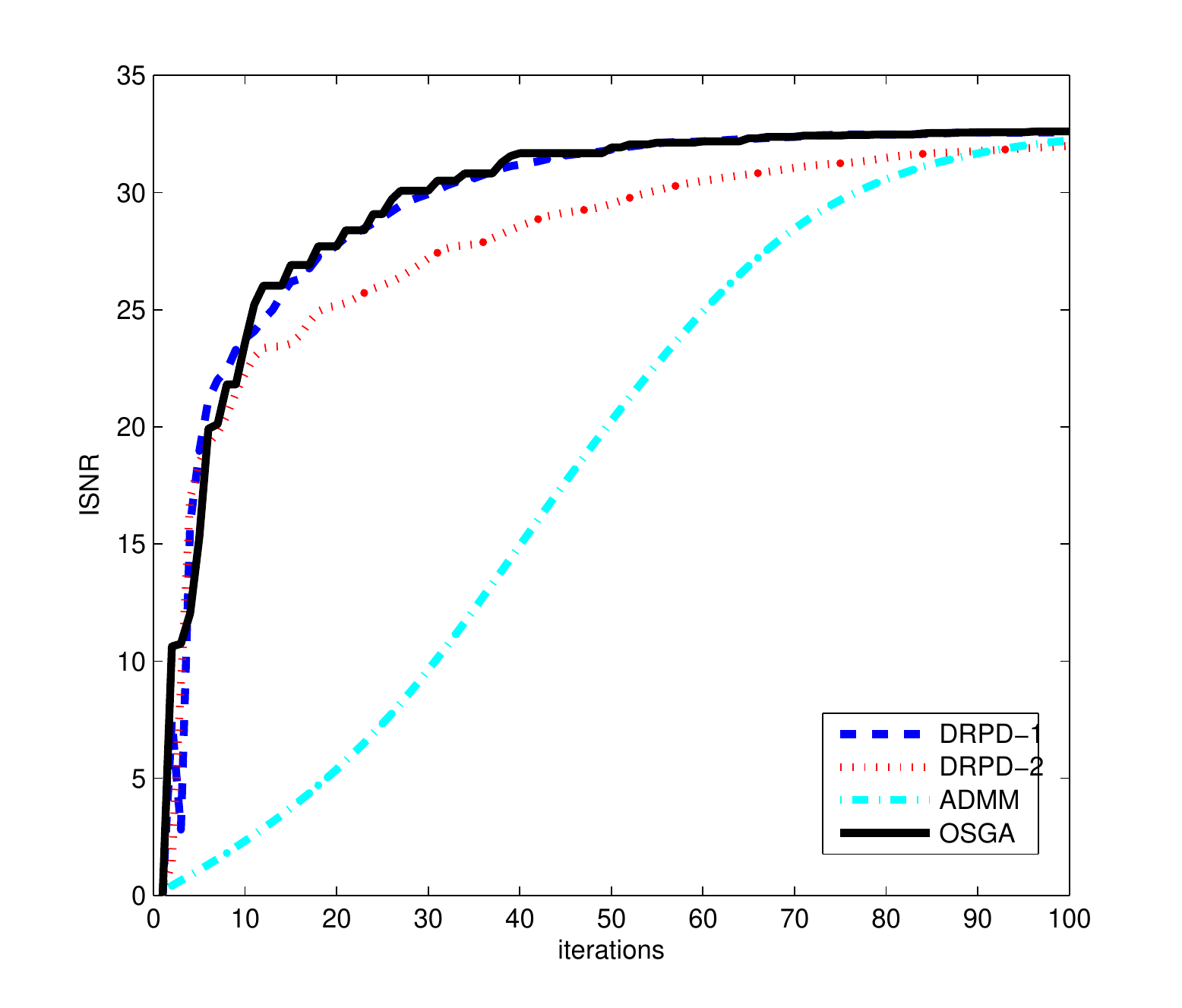}}
\qquad
\subfloat[][$\delta_k$ versus iterations, $\lambda = 5 \times 10^{-2}$]{\includegraphics[width=6.1cm]{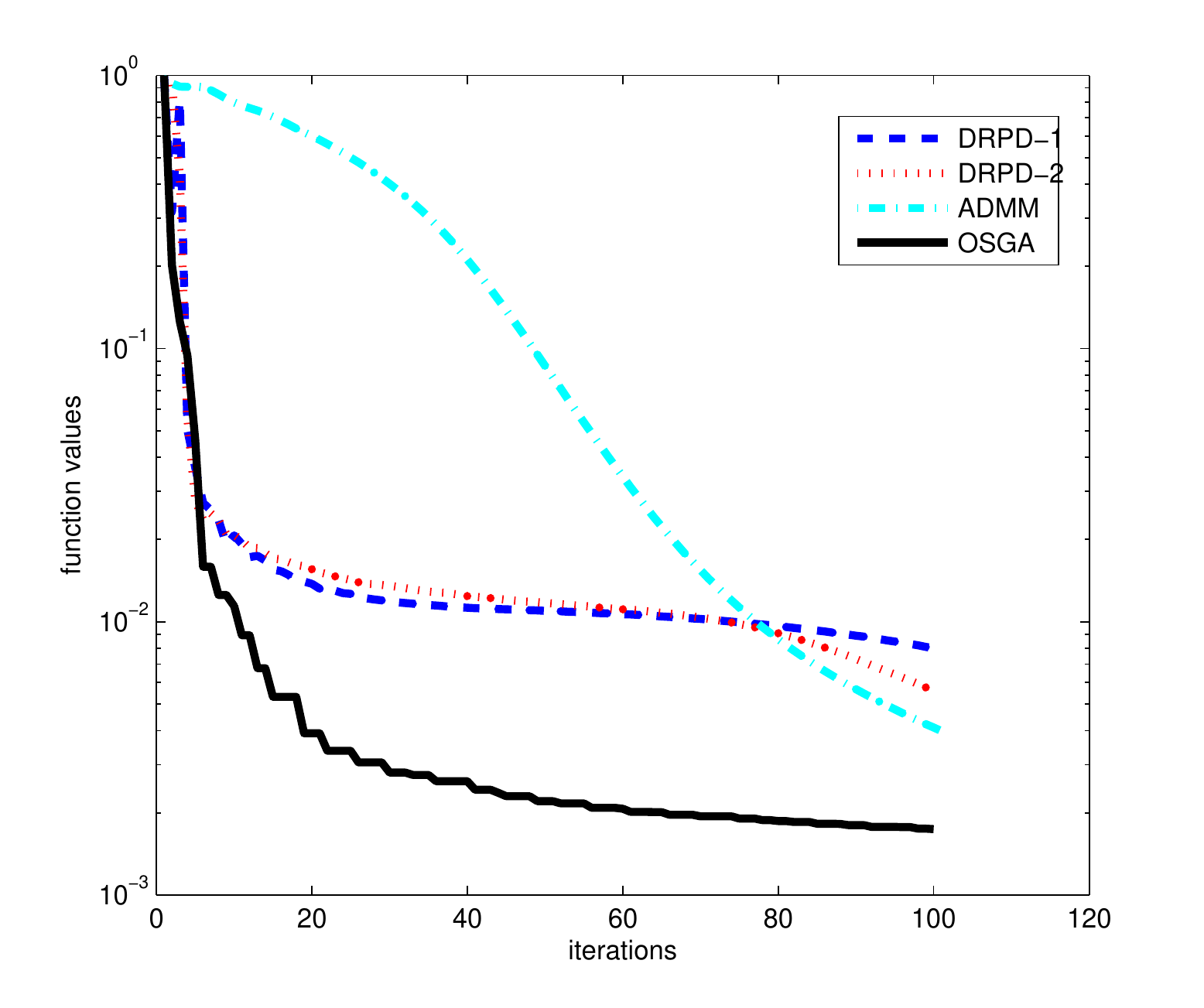}}%
\qquad
\subfloat[][ISNR versus iterations, $\lambda = 5 \times 10^{-2}$]{\includegraphics[width=6.1cm]{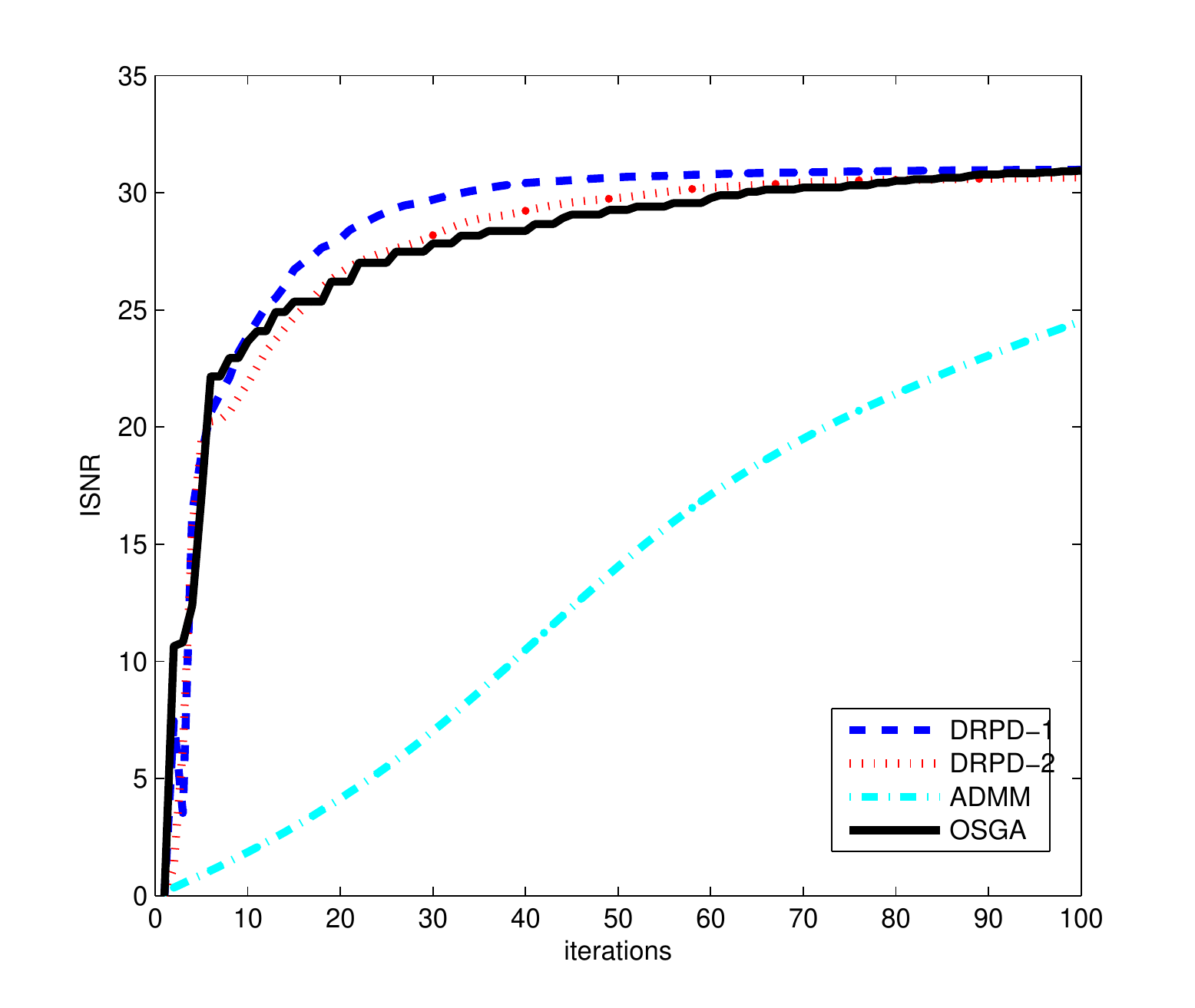}}

\caption{A comparison among DRPD-1, DRPD-2, ADMM, and OSGA for deblurring the $641 \times 641$ Dione image with the various regularization parameter $\lambda$. The blurred/noisy image was constructed by the $7 \times 7$ Gaussian kernel with standard deviation 5 and salt-and-pepper impulsive noise with the level $50 \%$. The algorithms were stopped after 100 iterations. Subfigures (a), (c), and (e) display the relative error of function values $\delta_k$ (\ref{e.delta}) versus iterations, and (b), (d), and (f) demonstrate ISNR (\ref{e.isnr}) versus iterations.}
\end{figure}

\begin{figure} \label{f.deb4}
\centering
\subfloat[][Original image]{\includegraphics[width=6.1cm]{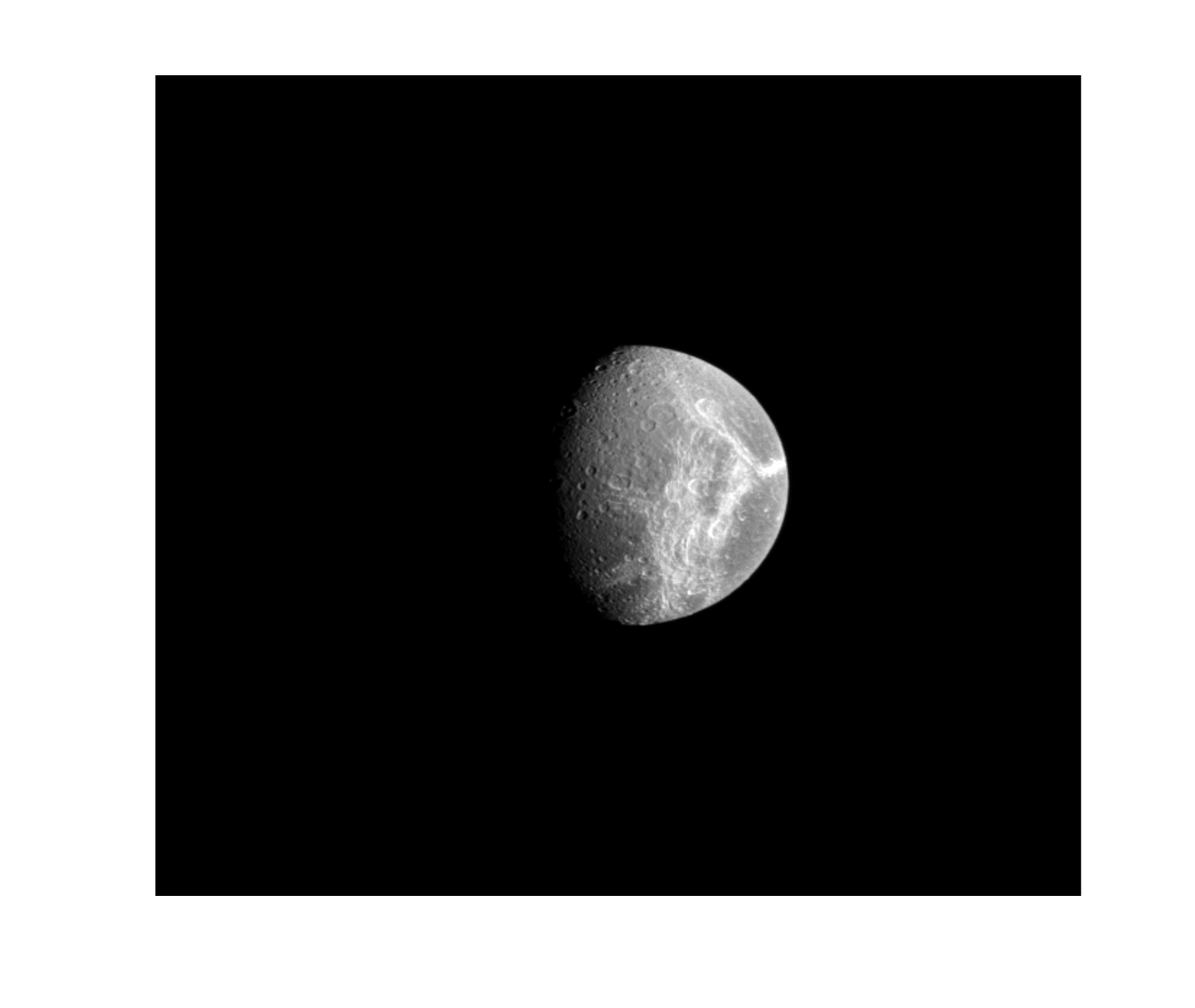}}%
\qquad
\subfloat[][Blurred/noisy image]{\includegraphics[width=6.1cm]{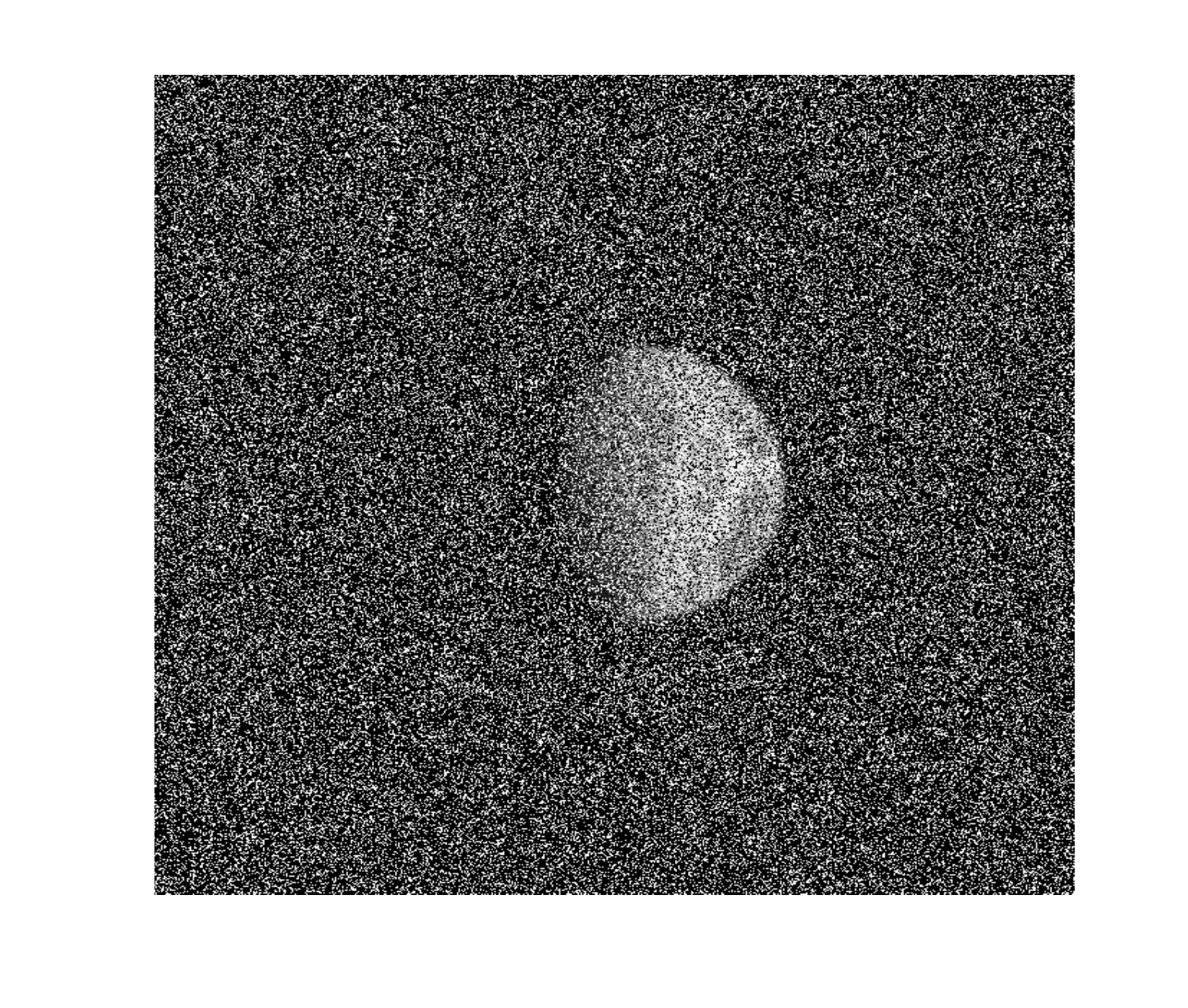}}
\qquad
\subfloat[][DRPD-1: $f = 1.0324e+5, \mathrm{PSNR} = 38.70, \mathrm{T} = 10.43$]{\includegraphics[width=6.1cm]{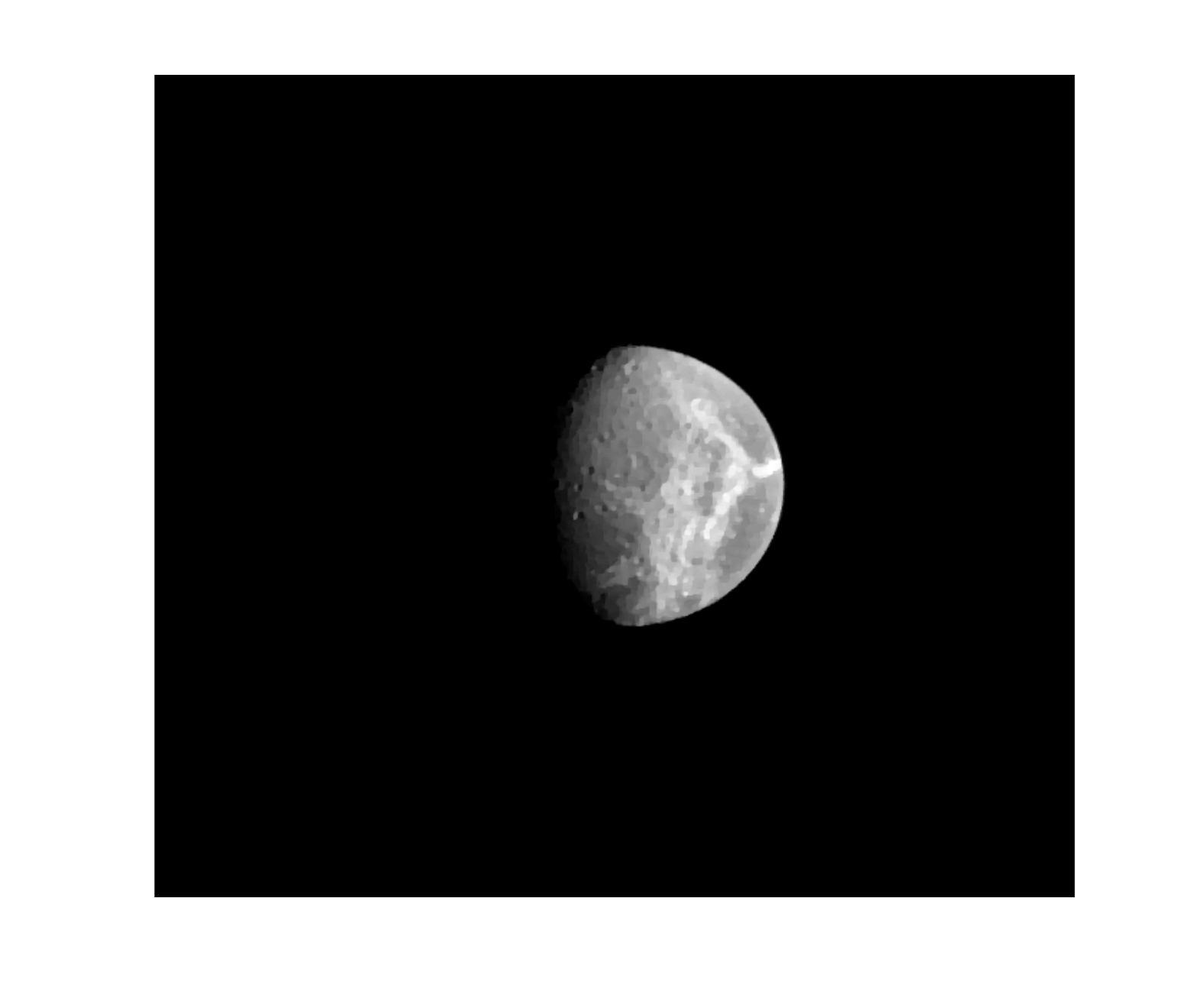}}%
\qquad
\subfloat[][DRPD-2: $f = 1.0294e+5, \mathrm{PSNR} = 38.11, \mathrm{T} = 6.68$]{\includegraphics[width=6.1cm]{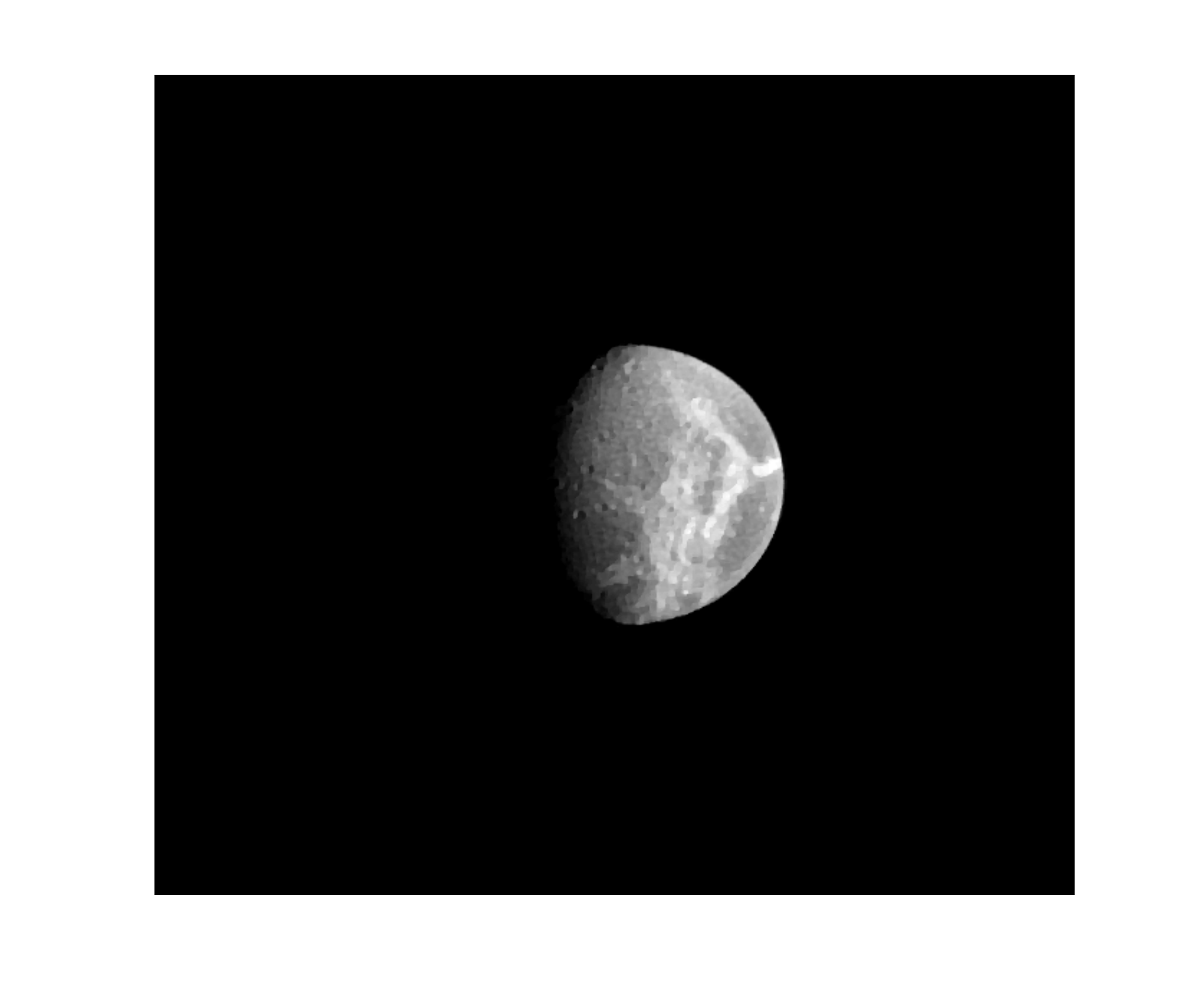}}
\qquad
\subfloat[][ADMM: $f = 1.0281e+5, \mathrm{PSNR} = 38.35, \mathrm{T} = 8.46$]{\includegraphics[width=6.1cm]{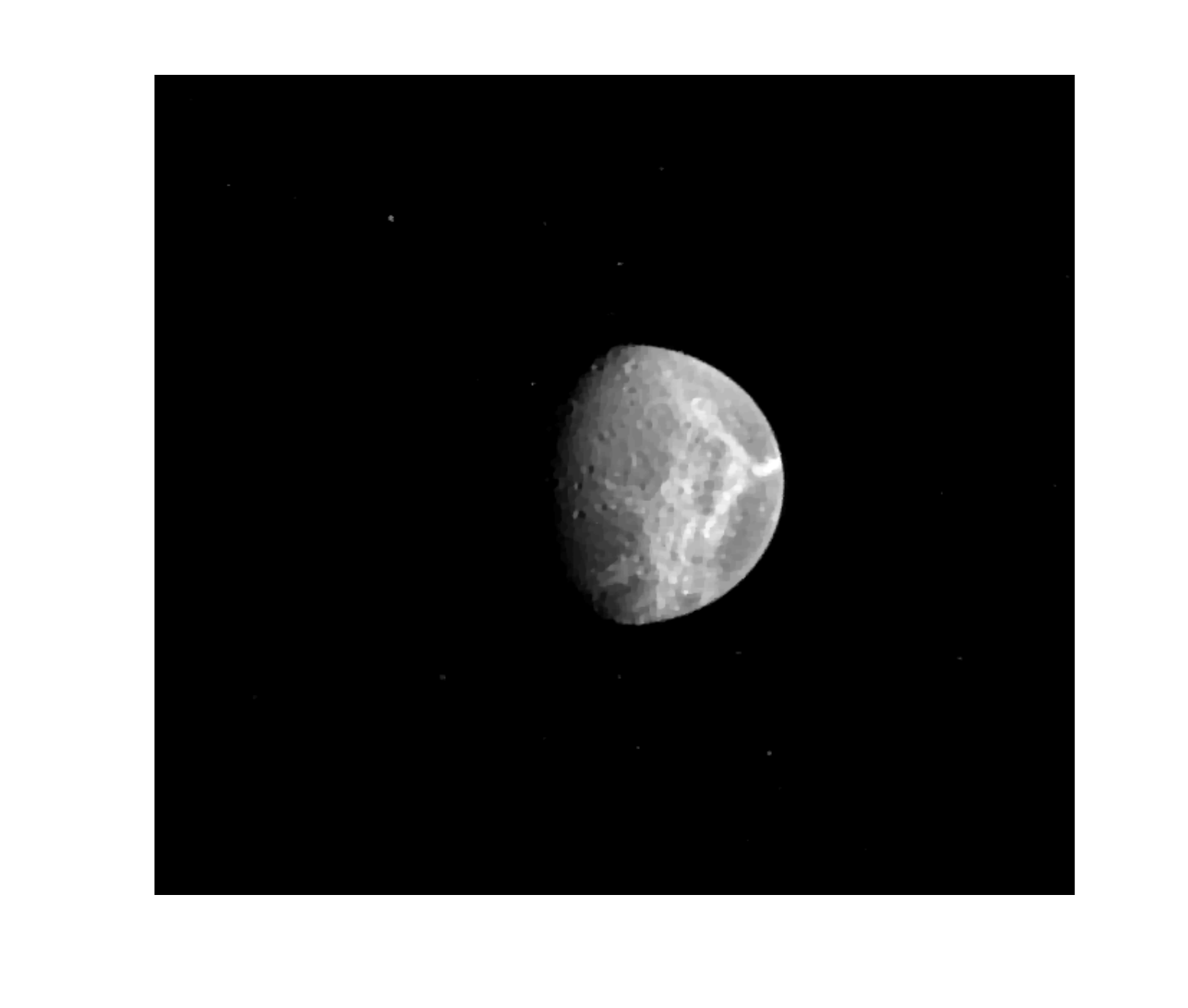}}%
\qquad
\subfloat[][OSGA: $f = 1.0281e+5, \mathrm{PSNR} = 38.73, \mathrm{T} = 8.32$]{\includegraphics[width=6.1cm]{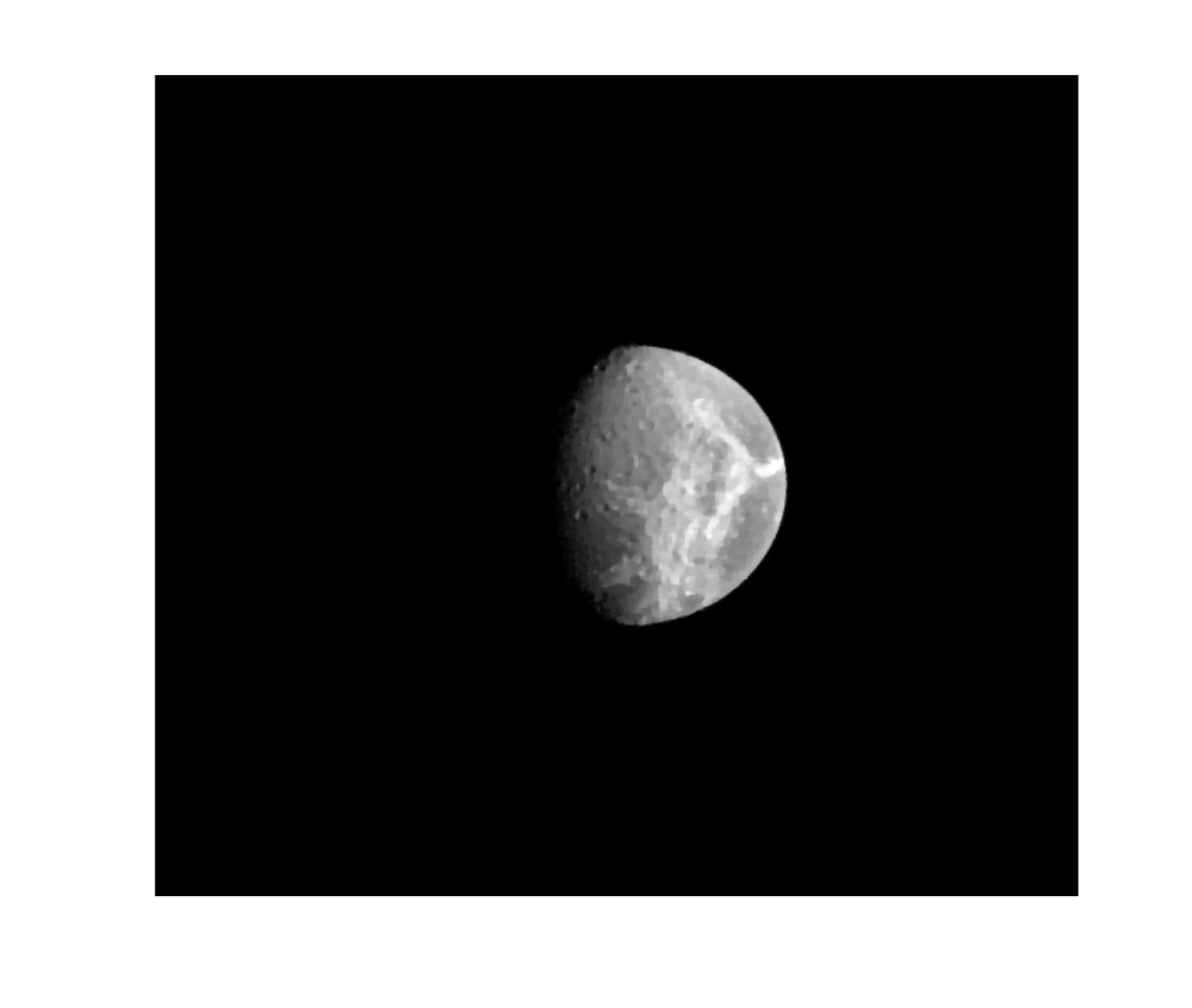}}
\caption{Deblurring of the $641 \times 641$ Dione image using DRPD-1, DRPD-2, ADMM and OSGA with the parameter $\lambda = 10^{-1}$. The algorithms were stopped after 100 iterations. The blurred/noisy image was constructed by the $7 \times 7$ Gaussian kernel with standard deviation 5 and salt-and-pepper impulsive noise with the level $50 \%$.}%
\end{figure}

\section{Conclusions}
In this paper an optimal subgradient method, OSGA, is addressed for solving structured convex constrained optimization. More specifically, finding a solution of OSGA's subproblem is investigated in the presence of some convex constraints. Two types of convex constraints are considered, namely, simple convex domains, in which the orthogonal projection in the domains is effectively available, and functional constraints, defined as the sublevel sets of simple convex functions. In each case some interesting examples are discussed for which OSGA's subproblem can be solved efficiently. Numerical results and comparisons with some state-of-the-art algorithms are reported showing that OSGA is efficient and reliable for solving convex optimization problems in applications. \\\\\\
{\bf Acknowledgement.} We would like to thank {\sc Radu Bot} and {\sc Min Tao} for making their codes DRPD-1, DRPD-2, and ADMM available for us. 


\end{document}